\documentclass{amsart}
\usepackage{amsmath,amsfonts,amssymb,amsthm,}
\usepackage{mathrsfs,amscd,comment,scalefnt}
\usepackage{upgreek}

\newcommand\ScaleExists[1]{\vcenter{\hbox{\scalefont{#1}$\exists$}}}

\newcommand\ScaleForall[1]{\vcenter{\hbox{\scalefont{#1}$\forall$}}}

\DeclareMathOperator*\bigexists{%
  \vphantom\sum
  \mathchoice{\ScaleExists{1.8}}{\ScaleExists{1.2}}{\ScaleExists{0.9}}{\ScaleExists{0.75}}}
  
\DeclareMathOperator*\bigforall{%
  \vphantom\sum
  \mathchoice{\ScaleForall{1.8}}{\ScaleForall{1.2}}{\ScaleForall{0.9}}{\ScaleForall{0.75}}}

\begin{document}

\newcommand\hide[1]{\commented{gray}{Hidden:}{#1}}
\renewcommand\hide[1]\empty

\theoremstyle{plain}
  \newtheorem{thm}{Theorem}
  \newtheorem{prop}{Proposition}
  \newtheorem{lmm}{Lemma}
  \newtheorem{coro}{Corollary}
  \newtheorem*{tm}{Theorem}
  \newtheorem*{cor}{Corollary}
  \newtheorem*{lm}{Lemma}
  \newtheorem*{prp}{Proposition}

\theoremstyle{definition}
  \newtheorem*{df}{Definition}
  \newtheorem{prb}{Problem}
  \newtheorem{q}{Question}
  \newtheorem{rmk}{Remark}
  \newtheorem{rmks}{Remarks}
  \newtheorem*{rfr}{References}

  \newtheorem*{usaxioms}{Usual Axioms}
  \newtheorem*{conjaxiom}{The $\bigwedge$-Axiom}
  \newtheorem*{usrules}{Usual Rules}
  \newtheorem*{conjrule}{The $\bigwedge$-Rule}
  \newtheorem*{forallrule}{The $\bigforall$-Rule}
  \newtheorem*{S-rule}{The $S$-Rule}
  \newtheorem*{A-rule}{The $\mathcal A$-Rule}
  \newtheorem*{Theta-rule}{The $\Theta$-Rule}
  \newtheorem*{inflax}{The Inflation Axiom}
  \newtheorem*{inflrule}{The Inflation Rule}
  \newtheorem*{consax}{The $\Con$-Axiom}
  \newtheorem*{consrule}{The $\Con$-Rule}
  \newtheorem*{rfrule}{The Reflection Rule}

\theoremstyle{remark}
  \newtheorem*{exm}{Example}
  \newtheorem*{exms}{Examples}
  \newtheorem*{fct}{Fact}
  \newtheorem*{fcts}{Facts}
  \newtheorem*{atomic}{Atomic formulas}
  \newtheorem*{connections}{Connections}
  \newtheorem*{quantifiers}{Quantifiers}
  \newtheorem*{case1}{\;\;\;\;\;Case 1}
  \newtheorem*{case2}{\;\;\;\;\;Case 2}
  \newtheorem*{negation}{Negation}
  \newtheorem*{conjunction}{Conjunction}
  \newtheorem*{universal quantifier}{Universal quantifier}
  \newtheorem*{iif}{If}
  \newtheorem*{onlyif}{Only if}

\renewcommand\labelenumi{\textnormal{(\roman{enumi})}}
\renewcommand\labelenumii{\textnormal{(\alph{enumii})}}

\newcommand{\lra}{\leftrightarrow}

\newcommand{\image}{\/``\,}
\newcommand{\cf}{ {\mathop{\mathrm {cf\,}}\nolimits} }
\newcommand{\tc}{ {\mathop{\mathrm {tc\,}}\nolimits} }
\newcommand{\dom}{ {\mathop{\mathrm {dom\,}}\nolimits} }
\newcommand{\ran}{ {\mathop{\mathrm {ran\,}}\nolimits} }
\newcommand{\Con}{ {\mathop{\mathrm {Con\,}}\nolimits} }
\newcommand{\rdc}{ {\mathop{\mathrm {rd\,}}\nolimits} }
\newcommand{\prv}{ {\mathop{\mathrm {Prv\,}}\nolimits} }
\newcommand{\pr}{ {\mathop{\mathrm {Pr\,}}\nolimits} }
\newcommand{\fv}{ {\mathop{\mathrm {fv\,}}\nolimits} }
\newcommand{\cmpl}{ {\mathop{\mathrm {cmpl\,}}\nolimits} }
\newcommand{\pd}{ {\mathop{\mathrm {pd\,}}\nolimits} } 

\newcommand{\ulc}{\ulcorner} 
\newcommand{\urc}{\urcorner}

\newcommand{\PA}{ {\mathrm {PA}} }
\newcommand{\TA}{ {\mathrm {TA}} }
\newcommand{\KP}{ {\mathrm {KP}} }
\newcommand{\Z}{ {\mathrm Z} }
\newcommand{\ZF}{ {\mathrm {ZF}} }
\newcommand{\ZFA}{ {\mathrm {ZFA}} }
\newcommand{\ZFC}{ {\mathrm {ZFC}} }
\newcommand{\NBG}{ {\mathrm {NBG}} }

\newcommand{\E}{ {\mathrm {AE}} }
\newcommand{\AR}{ {\mathrm {AR}} }
\newcommand{\WR}{ {\mathrm {WR}} }
\newcommand{\AF}{ {\mathrm {AF}} }
\newcommand{\AC}{ {\mathrm {AC}} }
\newcommand{\DC}{ {\mathrm {DC}} }
\newcommand{\GWO}{ {\mathrm {GWO}} }
\newcommand{\AD}{ {\mathrm {AD}} }
\newcommand{\PD}{ {\mathrm {PD}} }
\newcommand{\CH}{ {\mathrm {CH}} }
\newcommand{\GCH}{ {\mathrm {GCH}} }
\newcommand{\AInf}{ {\mathrm {AInf}} }
\newcommand{\AU}{ {\mathrm {AU}} }
\newcommand{\AP}{ {\mathrm {AP}} }
\newcommand{\APr}{ {\mathrm {APr}} }
\newcommand{\ASp}{ {\mathrm {ASp}} }
\newcommand{\ARp}{ {\mathrm {ARp}} }
\newcommand{\AFA}{ {\mathrm {AFA}} }
\newcommand{\BAFA}{ {\mathrm {BAFA}} }
\newcommand{\FAFA}{ {\mathrm {FAFA}} }
\newcommand{\SAFA}{ {\mathrm {SAFA}} }
\newcommand{\T}{\mathrm {T}}
\newcommand{\Th}{\mathrm {Th}}
\newcommand{\SSS}{\mathrm {S}}

\newcommand{\Rfn}{\mathrm{Rfn}}
\newcommand{\RFN}{\mathrm{RFN}}

\newcommand{\Val}{\mathrm{Val}}
\newcommand{\Dg}{\mathrm{Dg}}
\newcommand{\CK}{\mathrm{CK}}
\newcommand{\GLP}{ {\mathrm {GLP}} }
\newcommand{\OD}{ {\mathrm {OD}} }
\newcommand{\HOD}{ {\mathrm {HOD}} }

\pagestyle{headings}

\author{Denis~I.~Saveliev}
\title{A~note on restriction rules}
\date{}

\thanks{
\noindent
\!\!\!\!\!\!\!\!\!\!\!\!
{\em MSC~2010}:
Primary 
03C35, 
03F03 
Secondary 
03C40, 
03C50, 
03C70, 
03C75, 
03C85, 
03E45, 
03F35, 
03F45. 
}


\thanks{
\noindent
{\em Acknowledgement}:
Partially supported by Russian Science Foundation 
grant 21-11-00318.	
}

\hide{
\begin{abstract}
We proffer a~class of rules of inference, called here 
restriction rules, allowing to deduce complete theories 
of given models.
E.g., the Zermelo--Fraenkel set theory in the logic 
expanded by the corresponding restriction rule gives 
the complete theory of $L$, the constructible universe. 
Similar rules can be provided for higher-order and 
infinitary logics, with the same effect.
\end{abstract}
}

\begin{abstract}
We consider a~certain class of infinitary rules 
of inference, called here restriction rules, using 
of which allows us to deduce complete theories of 
given models. The first instance of such rules was the 
$\omega$-rule introduced by Hilbert, and generalizations 
of the $\omega$-rule were first considered by Henkin. 
Later on Barwise showed that within countable languages, 
for any countable model~$\mathfrak M$, first-order logic 
expanded by the corresponding $\mathfrak M$-rule deduces 
from the diagram of $\mathfrak M$ all formulas that are 
true in all models that include~$\mathfrak M$, 
in particular, it deduces the (relativized) complete 
theory of~$\mathfrak M$. 
We show that, if the aim is only deducing 
the complete theory of a~given model, these countability 
assumptions can be omitted. Moreover, similar facts 
hold for infinitary and higher-order logics, 
even if these logics are highly incomplete. 
Finally, we show that Barwise's theorem in its 
stronger form, for vocabularies of arbitrary cardinality, 
holds for infinitary logics $\mathscr L_{\kappa,\lambda}$ 
whenever $\kappa$~is supercompact. 
The note also contains brief historical remarks.
\end{abstract}

\maketitle


\section{Introduction}

{\em True arithmetic}~$\TA$ is the complete 
first-order theory of $(\omega,0,s,+,\,\cdot\,,<)$, 
the standard model of Peano arithmetic~$\PA$. 
By the classical undefinability theorem of Tarski, 
$\TA$~is not arithmetically definable. A~sharper 
version of this result is Post's theorem stating that 
the $\Sigma^{0}_{n}$-fragment of $\TA$ is arithmetically 
definable but only by a~formula of complexity higher 
than~$\Sigma^{0}_{n}$. 

Since by G{\"o}del's second incompleteness theorem,
no recursively enumerable extension~$\T$ of $\PA$ 
proves even its own consistency $\Con(\T)$, which 
is a~$\Pi^{0}_{1}$-formula, a~fortiori there is no 
effective way to describe~$\TA$. Among non-effective 
ways to do this, one approach is to expand the usual 
first-order finitary language $\mathscr L_{\omega,\omega}$; 
the second-order language, or even the first-order 
infinitary language $\mathscr L_{\omega_1,\omega}$ 
with countable connectives, suffices to describe $\TA$ 
by a~single formula (see, e.g.,~\cite{Barwise Feferman}). 
Another approach is, remaining within the language 
$\mathscr L_{\omega,\omega}$, to extend the usual set 
of inference rules with a~schema of infinite rules, 
called the $\omega$-rule.

Recall that, for each formula $\varphi(x)$ in 
one parameter~$x$, the {\em $\omega$-rule} infers 
the formula $\forall x\,\varphi(x)$ from the countable 
set of formulas $\varphi(t_n)$, where $t_n$~is 
the term of form $s(s(\ldots(0)\ldots))$ with $n$~times 
applied~$s$ (interpreted by the successor operation):
$$
\dfrac
{\{\varphi(t_n):n\in\omega\}}
{\forall x\,\varphi(x)}.
$$
The $\omega$-rule first appeared in Hilbert's 
paper~\cite{Hilbert 1930}.%
\footnote{
Possibly the rule was mentioned in Tarski's 
``unpublished talk'', 1927, 
see \cite{Sundholm 1983} for historical details. 
}
Later on it was investigated by Rosser 
in~\cite{Rosser 1937}, who coined the name 
{\em Carnap's rule} that has since been 
occasionally found 
(Carnap mentioned the rule in~\cite{Carnap 1935}, 
p.~173, bottom), and then used in works 
of Novikov~\cite{Novikov 1943}, 
Sch{\"u}tte~\cite{Schutte 1951}, et~al.  
Shoenfield observed that a~constructive version of 
the $\omega$-rule suffices to deduce $\TA$ from~$\PA$ 
(\cite{Shoenfield}; see also, e.g.,~\cite{Franzen}). 
These investigations were developed by Orey~\cite{Orey 1956} 
and by Henkin \cite{Henkin 1954,Henkin 1957}, who 
considered generalizations of the $\omega$-consistency 
and the $\omega$-completeness; later on Barwise considered 
a~relativized variant of the rule and establish 
a~corresponding completeness result for arbitrary 
countable models of the countable language 
(\cite{Barwise 1975}, p.~87 onwards); 
see below for more on these results. 
A~few concrete instances of generalized $\omega$-rules 
can be found in the literature; these are, e.g., 
a~rule for a~non-standard model of $\PA$ instead of 
its standard model used in~\cite{Kotlarski etal}; 
a~rule for $\omega^{\mathrm{CK}}_1$, the least 
non-recursive ordinal, instead of~$\omega$ 
used in~\cite{Pohlers}. 

This note provides a~small contribution to the area. 
The main observation here is that, if our aim is 
only to deduce the complete theory of a~given model, 
the countability assumptions of~\cite{Barwise 1975} 
can be omitted. Moreover, analogous facts hold for 
infinitary and higher-order logics as well, even if 
these logics are highly incomplete.

The deductive systems considered here differ from those 
that were used by Barwise result only in minor details. 
As said, we consider rules with arbitrarily many 
premises and use them in languages of arbitrary 
cardinalities. We call them {\em restriction rules} since 
they restrict the class of models of a~given theory in 
the same way as the $\omega$-rule restricts the class 
of models of $\PA$ to the models elementarily equivalent 
to the standard model; namely, the resulting complete 
theories describe pointwise definable models. 
Unlike Barwise, instead of using constants 
for elements of a~model under consideration, we use terms 
(as in the original $\omega$-rule) or even, most generally, 
formulas that uniquely characterize the elements; the 
notion of the diagram of a~model generalizes accordingly.
Also unlike Barwise, we do not treat the generalized diagram 
as new logical axioms but rather as axioms of a~basic theory 
designed to derive the complete theory of the model from it.


For simplicity, we start with the definition of restriction 
rules in the term version, which gives us the $\omega$-rule 
as an obvious instance. We show that, for a~set $S$ of terms, 
whenever a~theory decides whether an atomic formula with 
terms in $S$ substituting its free variables is true or not, 
then the theory becomes complete if we add the corresponding 
$S$-rule to the usual rules of first-order logic 
(Theorem~\ref{t: syn complet}). Based on this, we establish
a~semantic completeness result stating that a~suitable 
$S$-rule allows one to deduce the complete theory 
$\Th(\mathfrak A)$ of a~given model~$\mathfrak A$ 
from an appropriate, diagram-like fragment of the theory
(Theorem~\ref{t: sem complet}).

Further, we provide the general version of restriction 
rules using uniqueness formulas and note that the term 
version and the general version are essentially equivalent. 
Then we formulate the general relativized versions of 
two previous completeness results, syntactic 
(Theorem~\ref{t: syn complet gen}) and semantic
(Theorem~\ref{t: sem complet gen}), where the latter 
relates to the theories of pointwise definable models. 
As an example, we consider an extension of $\ZFC$, 
Zermelo--Fraenkel set theory, endowed with the 
corresponding restriction rule and show that it gives 
the complete theory of~$L$, G{\"o}del's universe of 
constructible sets (Proposition~\ref{p: rel thy of L}). 
Finally, we point out that restriction rules naturally 
lead to certain polymodal logics in the same way as 
the $\omega$-rule leads to the well-known logic~$\GLP$.


In the last section, we extend this apparatus to 
higher logics. Mainly we consider first-order infinitary 
languages $\mathscr L_{\kappa,\lambda}$. As the 
corresponding restriction rules deal with formulas 
in~$<\lambda$ parameters, the universal quantifiers 
in their conclusions are also $({<}\lambda)$-ary. 
We show that the basic deductive system 
(consisting of immediate generalizations of the usual 
logical axioms and rules) strengthened with a~suitable 
restriction rule provides analogous syntactic and 
semantic completeness results as in the finitary case 
(Theorems \ref{t: syn complet inf} 
and~\ref{t: sem complet inf}), even if this logic 
without the restriction rule is highly incomplete. 
Moreover, we consider the deductive system that includes 
additional axioms and rules (related to distributivity 
and choice) and show that if $\kappa$~is supercompact 
then Barwise's theorem holds in 
$\mathscr L_{\kappa,\lambda}$ 
in a~stronger form, with vocabularies of arbitrary 
cardinality (Theorem~\ref{t: Barwise for supercompact}). 
Finally, we briefly discuss the case of second-order 
(finitary) languages and their appropriate restriction 
rules, consider the derivation of true second-order 
arithmetic as an example, and point out that this 
machinery extends to higher-order (even infinitary) 
languages.

The note also contains short overview on earlier 
results and poses some questions.

Let us remark that mathematical side of this note 
is quite easy and standard; our purpose is rather 
to present the idea of restriction rules allowing us 
to get the complete theories of arbitrary models, with 
the expectation that its further development may lead 
to finer constructions.


\section{Restriction rules, the term form}

In this and the next sections, we work within 
$\mathscr L_{\omega,\omega}$, the usual first-order 
finitary language. Fix a~vocabulary~$\tau$. 
Let $S$ be a~set of closed terms in~$\tau$. 
The {\em restriction rule given by~$S$}, or
the $S$-{\em rule} for short, is the schema of rules 
that, for each formula $\varphi$ in a~single 
free variable~$x$, states:
$$
\dfrac
{\{\varphi(t):t\in S\}}
{\forall x\,\varphi(x)}.
$$
\vskip+0.3em
Note that the $S$-rule is infinitary whenever $S$~is 
infinite.

We denote by $\vdash$ the provability generated by all 
standard axioms and rules of classical first-order logic, 
as well as any stronger provability, and by $\vdash_S$ 
the (stronger) provability generated by $\vdash$ and 
the $S$-rule.

Given a~theory~$\T$ and a~provability~$\vdash$,  
the system $(\T,\vdash)$ is: 
\begin{itemize}
\setlength\itemsep{0.2em}
\item[(i)] 
{\it consistent\/} iff 
there is no~$\varphi$ such that
$\T\vdash\varphi$ and $\T\vdash\neg\,\varphi$;
\item[(ii)] 
{\it complete\/} iff for every~$\varphi$, 
$\T\vdash\varphi$ or $\T\vdash\neg\,\varphi$;  
\item[(iii)] 
$S$-{\it consistent\/} iff
there is no~$\varphi$ such that
$\T\vdash\varphi(t)$ for all $t\in S$
and $\T\vdash\exists x\,\neg\,\varphi(x)$;
\item[(iv)] 
$S$-{\it complete\/} iff
for every~$\varphi$,
if $\T\vdash\varphi(t)$ for all $t\in S$
then $\T\vdash\forall x\,\varphi(x)$.
\end{itemize}

\hide{
[The name of the $S$-completeness is somewhat confusing 
as it does not provide a~completeness... Essentially, 
$(\T,\vdash)$ is $S$-complete iff it is closed under 
the $S$-rule, i.e., $[\T,\vdash]=[\T,\vdash_S]$.]
}

\begin{lmm}\label{l: triv}
For any $\T$ and $\vdash$, 
\begin{itemize}
\setlength\itemsep{0.2em}
\item[(i)] 
if $(\T,\vdash)$ is $S$-consistent,
then it is consistent;
\item[(ii)] 
if $(\T,\vdash)$ is $S$-consistent and complete,
then it is $S$-complete;
\item[(iii)] 
if $(\T,\vdash)$ is consistent and $S$-complete,
then it is $S$-consistent;
\item[(iv)] 
$(\T,\vdash_S)$ is $S$-complete.
\end{itemize}
\end{lmm}

\begin{proof}
Straightforward from definitions.
\end{proof}



The notions of the $S$-consistency and the 
$S$-completeness for arbitrary sets $S$ of constants 
appeared in Henkin's papers 
\cite{Henkin 1954,Henkin 1957}, respectively. 
Before moving on, we shall very briefly outline 
the relevant works of Henkin, Orey, and Barwise; 
for a~bit more on them, see the Appendix.%
\footnote{Cf.~also Forster's commentary~\cite{Forster} 
on \cite{Henkin 1954,Henkin 1957}.}

By~\cite{Henkin 1954}, $\T$~is $S$-{\em satisfiable} 
(provided the vocabulary of $\T$ contains~$S$) iff 
there is a~model satisfying $\T$ the universe of which 
consists of interpretations of the constants in~$S$ 
(i.e., omitting the $1$-type $\{x\ne t:t\in S\}$). 
The following facts were stated in~\cite{Henkin 1954}: 
any $S$-satisfiable $\T$ is $S$-consistent, but not 
conversely unless $|S|<\omega$; stronger versions of 
the $S$-consistency (defined via formulas in $n$~free 
variables) form a~non-degenerate hierarchy, they are 
all implied by the $S$-satisfiability but still not 
sufficient to imply it; finally, a~notion of the 
{\em strong $S$-consistency} (defined uniformly for 
all formulas) gives exactly the $S$-satisfiability.

By~\cite{Henkin 1957}, 
$\T$~is $S$-{\em saturated}\,%
\footnote{Warning: do not confuse with saturated models!}
iff all sentences that are true in all models 
$S$-satisfying~$\T$ are deduced from $\T$ 
(in the usual first-order logic). 
The following facts were stated in~\cite{Henkin 1957}: 
any $S$-saturated~$\T$ is $S$-complete; the converse 
implication holds within languages in countable 
vocabularies whenever $|S|\le\omega$, but not generally; 
finally, a~notion of the {\em strong $S$-completeness} 
(defined uniformly for all pairs of sentences and formulas 
in one parameter) gives exactly the $S$-saturatedness.

Related ideas were given by Orey. By~\cite{Orey 1956}, 
$\T$~is {\em $\omega(\alpha)$-consistent} iff it is 
closed under $\alpha$~applications of the $\omega$-rule; 
such theories were first considered by 
Rosser~\cite{Rosser 1937}. As shown in~\cite{Orey 1956}, 
$\T$~is $\omega(\omega_1)$-consistent iff it is 
$\omega$-satisfiable; this follows from results 
of~\cite{Henkin 1957} as well (we may also notice that 
$\omega_1$ can be replaced with $\omega^{\CK}_1$ here). 

Generalizing these results by Henkin and Orey, 
Barwise considered a~relativized version of the $S$-rule 
proving the following completeness result
(\cite{Barwise 1975}, Theorem~3.5): 
in a~countable first-order language,
if $\mathfrak A$ is a~countable model, $S$~consists of 
constants for its elements, and there is a~distinguished 
predicate for the universe of~$\mathfrak A$, then 
formulas deduced from a~given~$\T$ and the diagram of 
$\mathfrak A$ via the $S$-rule are exactly those that 
are true in all models of $\T$ including $\mathfrak A$
as a~submodel.

Essentially, Barwise's theorem is a~combination of 
Theorem~\ref{t: sem complet} for countable models 
with the omitting types theorem; hence the countability 
restrictions. If, however, we abandon the latter and 
keep only the aim to deduce the complete theory of 
a~given model, these restrictions can be eliminated, 
i.e., the model and the language can be taken to be 
arbitrarily large. We start by proving this fact for 
$\mathscr L_{\omega,\omega}$; later we shall see 
that analogous facts hold for infinitary and 
higher-order logics as well.


Let $[\T,\vdash]$ denote $\{\varphi:\T\vdash\varphi\}$. 
E.g., if $\vdash$~is the usual first-order provability, 
$[\PA,\vdash]$ is an incomplete theory while
$[\PA,\vdash_\omega]$ is~$\TA$, where $\vdash_\omega$ 
denotes of course $\vdash$ expanded by the $\omega$-rule.


\begin{thm}[Syntactic Completeness]\label{t: syn complet}
Let $\T$ be a~theory and $S$~a~set of closed terms.
Assume that for every atomic formula~$\varphi$,
$$
\T\vdash\bar\varphi
\;\text{ or }\;
\T\vdash\neg\,\bar\varphi
$$
whenever $\bar\varphi$ is the sentence obtained from 
$\varphi$ by a~substitution of some terms of~$S$ 
instead of all its free variables.
Then for every sentence~$\sigma$,
$$
\T\vdash_S\sigma
\;\text{ or }\;\
\T\vdash_S\neg\,\sigma,
$$
i.e., the theory $[\T,\vdash_S]$ is complete.
\end{thm}

\begin{proof}
Clearly, it suffices to verify that 
for every formula~$\psi$,
$$
\T\vdash_S\bar\psi
\;\text{ or }\;
\T\vdash_S\neg\,\bar\psi
$$
whenever $\bar\psi$ is obtained from~$\psi$
in the described way. Let us prove this 
by induction on construction of~$\psi$.

\begin{atomic}
By assumption (and since $\vdash_S$ includes~$\vdash$).
\end{atomic}
              
\begin{negation}
If the claim holds for~$\varphi$,
then it trivially holds for~$\neg\,\varphi$.
\end{negation}

\begin{conjunction}
Let $\psi$ be $\varphi_0\wedge\varphi_1$.
Fix some~$\bar\psi$; then $\bar\psi$ is 
$\bar{\varphi_0}\wedge\bar{\varphi_1}$ with the 
inherited $\bar{\varphi_0}$ and~$\bar{\varphi_1}$. 
We have two cases:

\begin{case1}
$\T\vdash_S\bar{\varphi_0}$ and 
$\T\vdash_S\bar{\varphi_1}$. Then
$\T\vdash_S\bar\psi$ by the conjunction rule.
\end{case1}

\begin{case2}
$\T\not\vdash_S\bar\varphi$
for some $\varphi\in\{\varphi_0,\varphi_1\}$. Then 
$\T\vdash_S\neg\,\bar\varphi$ (for this~$\bar\varphi$)
by inductive hypothesis. Therefore, since
$\neg\,\bar\varphi\to\neg\,\bar\psi$ is an instance 
of the conjunction axiom in contraposition,
$\T\vdash_S\neg\,\bar\psi$ by modus ponens.
\end{case2}
\end{conjunction}       

\begin{universal quantifier}
Let $\psi$ be $\forall x\,\varphi(x,v,\ldots)$.
Fix some~$\bar\psi$; then $\bar\psi$ is 
$\forall x\,\varphi(x,t,\ldots)$ where $\varphi$ 
has the free variable~$x$ and the inherited 
closed terms $t,\ldots\in S$. Denote this 
$\varphi(x,t,\ldots)$ by~$\chi$.
Again we have two cases:

\begin{case1}
$\T\vdash_S\bar\chi$ for all~$\bar\chi$.
Then $\T\vdash_S\bar\psi$ by the $S$-rule.
\end{case1}

\begin{case2}
$\T\not\vdash_S\bar\chi$ for some~$\bar\chi$.
Then $\T\vdash_S\neg\,\bar\chi$ (for this~$\bar\chi$)
by inductive hypothesis. Therefore,
$\T\vdash_S\neg\,\bar\psi$ by the generalization rule.
\end{case2}
\end{universal quantifier}

The theorem is proved. 
\end{proof}

Thus if a~$S$-complete $(T,\vdash)$ satisfies 
the assumption of Theorem~\ref{t: syn complet},
then it is complete.
Certainly, neither completeness nor $S$-completeness 
of $(T,\vdash_S)$ guarantees its consistency, even 
if $(T,\vdash)$ is consistent; e.g., let  
$T=\{\exists x\exists y\,(x\ne y)\}$ and $|S|=1$.

\vskip+1em

To get complete and consistent theories, we now 
turn to semantics. 

A~model~$\mathfrak A$ is an~$S$-{\em model} iff 
each $a\in A$ interprets some $t\in S$, i.e., iff 
$\mathfrak A$ omits the $1$-type $\{x\ne t:t\in S\}$.

\begin{lmm}\label{l: S-soundness}
The $S$-rule preserves validity in every $S$-model.
\end{lmm}

\begin{proof}
Immediate.
\end{proof}

Given a~model~$\mathfrak A$ with the universe~$A$ and 
$a\in A$, let us say that $t$~is a~{\em term for}~$a$ 
iff the set of points of~$\mathfrak A$ that are definable 
by the formula $x=t$ is the singleton~$\{a\}$.


\begin{thm}[Semantic Completeness]\label{t: sem complet}
Let $\mathfrak A$ be a~model, $S$ a~set of closed terms 
of form $t_a$ for each $a\in A$, 
and $\T$ a~theory (all in a~vocabulary~$\tau$).
Assume that 
$[\T,\vdash]\subseteq\Th(\mathfrak A)$ 
and 
$$
\T\vdash\bar\varphi
\;\text{ iff }\;
\mathfrak A\vDash\bar\varphi
$$
whenever $\varphi$~is an atomic or negated atomic formula 
and $\bar\varphi$~is obtained from it by a~substitution 
of some terms in $S$ instead of all its free variables. 
Then for all sentences~$\sigma$,
$$
\T\vdash_S\sigma
\;\text{ iff }\;
\mathfrak A\vDash\sigma,
$$
i.e., $[\T,\vdash_S]=\Th(\mathfrak A)$. 
\end{thm}

\begin{proof} 
The inclusion $[\T,\vdash_S]\subseteq\Th(\mathfrak A)$ 
follows from $[\T,\vdash]\subseteq\Th(\mathfrak A)$
by Lemma~\ref{l: S-soundness}, and the converse 
inclusion holds since the theory $[\T,\vdash_S]$ 
is complete by Theorem~\ref{t: syn complet}. 
\end{proof}

\hide{
[It seems, 
we can replace the condition 
$[\T,\vdash]\subseteq\Th(\mathfrak A)$ 
with a~slightly simpler condition 
$\T\subseteq\Th(\mathfrak A)$ 
by assuming that $\vdash$ is sound.]
}


As a~simple instance,
recall that the {\it diagram\/} 
of~$\mathfrak A$, denoted by~$\mathrm{Dg}(\mathfrak A)$,
is the theory in the language expanded by adding 
constants~$c_a$ for all $a\in A$ and consisting of 
all atomic sentences and their negations that are 
true in~$\mathfrak A$. Let $S=\{c_a:a\in A\}$. 
By Theorem~\ref{t: sem complet}, we have 
$[\mathrm{Dg}(\mathfrak A),\vdash_S]=\Th(\mathfrak A)$.

\begin{rmk}
Concerning the size of the set~$S$ suitable to deduce 
$\Th(\mathfrak A)$, note that if the vocabulary of 
$\mathfrak A$ is of cardinality~$\kappa$, then by the 
L{\"o}wenheim--Skolem downward theorem $\mathfrak A$ 
has an elementary submodel of 
cardinality~$\le\max(\omega,\kappa)$,
so we can pick $S$ with $|S|\le\max(\omega,\kappa)$. 
\end{rmk}


\subsection*{Relativized restriction rules}

Restriction rules can be given in a~more general, 
relativized form. In the case, the deduced complete 
theory turns out to be ``localized'', i.e., relativized;  
this method allows one to work within a~larger theory. 

If $S$ is a~set of closed terms and $U$ a~unary predicate 
symbol, the {\em $S$-rule relativized to~$U$} is the 
following schema of rules, for each formula $\varphi$ 
in one free variable:
$$
\dfrac
{\{\varphi^U(t):t\in S\}}
{(\forall x\,\varphi(x))^U}
$$
where, as usual, $\psi^U$ denotes the relativization 
of $\psi$ to~$U$. The provability generated by $\vdash$ 
and the $S$-rule relativized to $U$ will be denoted 
by~$\vdash^{U}_{S}$. Note that the former $S$-rule can 
be obtained from this relativized form by adding 
$\forall x\,U(x)$.

A~slight modification of the previous arguments proves 
the analogs of the above results; we do not write up 
them here because in the next section we shall give an 
even more general definition of restriction rules.

\hide{
\vskip+1em
[Notation:

Given a~model~$\mathfrak A$ and a~set $S=\{t_a:a\in A\}$ 
of closed terms $t_a$ for each $a\in A$, if we disregard  
syntactical details, we call the corresponding rule just 
the {\em (relativized) $\mathfrak A$-rule} and write 
$\vdash_\mathfrak A$ instead of $\vdash_S$ and 
$\vdash^\mathfrak A$ instead of $\vdash^{U}_S$.]
}



\section{Restriction rules, the general form}

Now we formulate restriction rules in a~general form, 
without assuming terms in the language under consideration.

If $\theta(x)$ is a~formula in one free variable~$x$,
let us say that $\theta$ is a~{\em uniqueness formula 
in $(\T,\vdash)$}
iff the system proves that 
there exists exactly one~$x$ satisfying~$\theta(x)$:
$$
\T\;\vdash\;
\exists x\,
(\theta(x)\wedge\,
\forall y\,(\theta(y)\to x=y)).
$$


Let $\Theta$ be a~set of uniqueness formulas in $(\T,\vdash)$.
The {\em restriction rule given by $\Theta$}, 
or the $\Theta$-{\em rule} for short, is the schema 
of rules that states, for each formula~$\varphi$ in 
a~free variable~$x$:
$$
\frac
{
\bigl\{\exists x\,(\varphi(x)\wedge\theta(x)):
\theta\in\Theta\bigr\}
}
{\forall x\,\varphi(x)}.
$$ 
(Note that, since $\Theta$ consists of uniqueness 
formulas, we can alternatively define the $\Theta$-rule 
by changing the formulas 
$\exists x\,(\varphi(x)\wedge\theta(x))$
in the premise to the formulas 
$\forall x\,(\theta(x)\to\varphi(x))$ 
without changing the resulting deductive system.)

Likewise, the {\em restriction rule given by $\Theta$ 
relativized to~$U$}, shortly, the {\em 
$(\Theta,U)$-rule}, is the schema 
$$
\frac
{\bigl\{(\exists x\,(\varphi(x)\wedge\theta(x)))^U:
\theta\in\Theta\bigr\}}
{(\forall x\,\varphi(x))^U}.
$$ 
The symbols $\vdash_\Theta$ and $\vdash^{U}_\Theta$
denote the provability generated by $\vdash$ and the 
$\Theta$-rule, respectively, the $(\Theta,U)$-rule. 


As it is easy to see, both proposed versions of restriction 
rules, one using terms and another using uniqueness 
formulas, express essentially the same and are mutually 
changeable in the following sense.

\begin{lmm}\label{l: term vs uniq}
Let $T$ be a~theory and $\vdash$ a~provability.
\begin{enumerate}
\item[(i)] 
For any set $S$ of closed terms there is a~set 
$\Theta(S)$ of uniqueness formulas such that 
$[\T,\vdash_S]=[\T,\vdash_{\Theta(S)}]$.
\item[(ii)] 
For any set $\Theta$ of uniqueness formulas 
there is a~set $S(\Theta)$ of closed terms (perhaps, 
in an expanded language) such that 
$[T,\vdash_\Theta]=[\T,\vdash_{S(\Theta)}]$.
\end{enumerate}
\end{lmm}

\begin{proof}
(i). 
If $\Theta(S):=\{x=t:t\in S\}$, then $\Theta(S)$ 
consists of uniqueness formulas, and we get:  
$\T\vdash_S\sigma$ iff 
$\T\vdash_{\Theta(S)}\!\sigma$,  
for all sentences~$\sigma$. 

(ii). 
Expand the language by adding the $\iota$-operator and 
put $S(\Theta):=\{\iota x\,\theta(x):\theta\in\Theta\}.$
Then $S(\Theta)$ consists of closed terms, and since 
using of $\iota$ gives a~conservative extension of 
the theory (see~\cite{Hilbert Bernays}), keeping for 
it the notation~$\T$, we get:
$\T\vdash_\Theta\sigma$ iff 
$\T\vdash_{S(\Theta)}\!\sigma$, 
for all $\iota$-free sentences~$\sigma$.
\end{proof}

\hide{
If $S$ consists of closed terms, put
$\Theta(S)=\{x=t:t\in S\}.$ Then $\Theta(S)$ 
consists of uniqueness formulas, and clearly
$$
\T\;\vdash_S\;\varphi
\;\text{ iff }\;
\T\;\vdash_{\Theta(S)}\;\varphi
$$
for every~$\varphi$.
In the opposite direction, if $\Theta$ consists 
of uniqueness formulas, expand the language by 
adding the $\iota$-operator and put
$
S(\Theta)=
\{\iota x\,\theta(x):\theta\in\Theta\}.
$
Then $S(\Theta)$ consists of closed terms, and 
since using of $\iota$ gives a~conservative extension 
of the theory (see~\cite{Hilbert Bernays}), it follows 
$$
\T\;\vdash_\Theta\;\varphi
\;\text{ iff }\;
\T\;\vdash_{S(\Theta)}\;\varphi
$$
for every $\iota$-free~$\varphi$.
}


The syntactic and semantic completeness results above
(Theorems \ref{t: syn complet} and~\ref{t: sem complet})
have obvious general versions using the $\Theta$-rule 
instead of terms.

\begin{thm}[Syntactic Completeness]\label{t: syn complet gen}
Let $\T$ be a~theory, $\vdash$ a~provability, 
$\Theta$ a~set of uniqueness formulas in $(\T,\vdash)$, 
and $U$~a~unary predicate symbol. Assume that 
\begin{enumerate}
\item[(i)] 
$\T\vdash\forall x\,(\theta(x)\to U(x))$
for all $\theta\in\Theta$, and 
\item[(ii)] 
for every atomic formula~$\varphi$ in $n+1$~variables 
and $\theta_0,\ldots,\theta_n\in\Theta$,
$$
\T\vdash\bar\varphi
\;\text{ or }\;
\T\vdash\neg\,\bar\varphi
$$
where $\bar\varphi$ is the sentence 
$
\exists x_0\ldots\exists x_n\,
(\varphi(x_0,\ldots,x_n)\wedge
\theta_0(x_0)\wedge\ldots\wedge\theta_n(x_n)).
$
\end{enumerate}
Then for every sentence~$\sigma$,
$$
\T\vdash^{U}_\Theta\sigma^{U}
\;\text{ or }\;\
\T\vdash^{U}_\Theta\neg\,\sigma^{U},
$$
i.e., the theory $\{\sigma^U:\T\vdash^{U}_\Theta\sigma\}$ 
is a~complete theory relativized to~$U$.
\end{thm}

(The predicate~$U$ may be new for $\T$, in which case 
we let (i) as an extra axiom, keeping the notation~$\T$ 
for the expansion.)

\begin{proof}
Modify the proof of Theorem~\ref{t: syn complet}.  
\end{proof}

In particular, we have the non-relativized variant:

\begin{coro}\label{c: syn complet gen}
Under the assumptions on $\T,\vdash,\Theta$ 
of Theorem~\ref{t: syn complet gen}, 
the theory $[\T,\vdash_\Theta]$ is complete. 
\end{coro}

\begin{proof}
Theorem~\ref{t: syn complet gen} for the universal~$U$ 
(or, alternatively, combine Theorem~\ref{t: syn complet} 
with Lemma~\ref{l: term vs uniq}(ii)).
\end{proof}


Given a~model~$\mathfrak A$, an element $a\in A$ is 
{\em definable without parameters over}~$\mathfrak A$ 
iff there is a~formula $\theta(x)$ in $1$~free 
variable~$x$ such that 
$$
\mathfrak A\vDash 
(\theta(x)\,\wedge\,
\forall y\,(\theta(y)\to x=y))
\;[a/x].
$$
Thus $a$~satisfies a~unique principal complete $1$-type. 
\hide{

Let us say, as usual, $(\T,\vdash)$ is 
\begin{enumerate}
\item[(i)] 
{\em (semantically) sound} iff 
$\T\vdash\varphi$ implies $\T\vDash\varphi$, 
\item[(ii)] 
{\em (semantically) complete} iff 
$\T\vDash\varphi$ implies $\T\vdash\varphi$.
\end{enumerate}
Let 
\begin{enumerate}
\item[(a)] 
$\theta$~is a~uniqueness formula in $(\T,\vdash)$,
\item[(b)] 
in any model of $\T$, $\theta$~defines a~unique element.
\end{enumerate}
Obviously, 
if $(\T,\vdash)$ is sound then (a)~implies~(b), and 
if $(\T,\vdash)$ is complete then (b)~implies~(a). 
E.g., for the usual (first-order finitary) 
provability~$\vdash$, (a) and (b) are equivalent. 
}
Clearly, if $\vdash$ is the usual (first-order finitary) 
provability, then by its completeness, 
$\theta$~is a~uniqueness formula in 
$(\T,\vdash)$ iff $\theta$~defines a~unique element 
of~$\mathfrak A$, for any model $\mathfrak A$ of~$\T$. 

A~model $\mathfrak A$ is {\em pointwise definable} iff 
each of its elements is definable without parameters. 
For an arbitrary $\mathfrak A$, its {\em pointwise 
definable part}, denoted by $\pd(\mathfrak A)$, consists 
of all elements of $\mathfrak A$ that are definable 
without parameters; clearly, these elements always form 
a~submodel, and $\mathfrak A$ is pointwise definable 
iff $\mathfrak A=\pd(\mathfrak A)$.

\begin{lmm}\label{l: pd}
For any models $\mathfrak A,\mathfrak B$ 
of a~vocabulary~$\tau$,
\begin{enumerate}
\item[(i)]
$|\pd(\mathfrak A)|\le\max(|\tau|,\omega)$,
\item[(ii)]
$\mathfrak A\equiv\mathfrak B$ implies
$\pd(\mathfrak A)\simeq\pd(\mathfrak B)$, 
\item[(iii)]
$\pd(\mathfrak A)\equiv\mathfrak A$ iff 
$\pd(\mathfrak A)\preceq\mathfrak A$ iff 
$\mathfrak A$ has definable Skolem functions.
\end{enumerate}
\end{lmm}


Obviously, any pointwise definable model is {\em prime}, 
i.e., it elementary embeds into any model of its theory. 
Thus a~prime model is as simple as possible: it realizes 
only the types that cannot be omitted.


Let $\mathfrak A$ be a~model that is pointwise definable 
(possibly, in an expanded language) and $\Theta$ 
a~set of uniqueness formulas defining all $a\in A$. 
Let us say that a~sentence is {\em $\Theta$-atomic}
iff it has the form
$
\exists x_0\ldots\exists x_n\,
(\varphi(x_0,\ldots,x_n)\wedge
\theta_0(x_0)\wedge\ldots\wedge\theta_n(x_n))
$
where $\varphi$ is an atomic formula, say, in 
$n+1$~variables, and $\theta_0,\ldots,\theta_n\in\Theta$. 
The {\it generalized diagram\/} of $\mathfrak A$ 
(given by~$\Theta$) is the theory consisting of all 
$\Theta$-atomic and negated $\Theta$-atomic sentences 
that are true in~$\mathfrak A$; we denote it by 
$\Dg_\Theta(\mathfrak A)$. The relativized concepts 
$
\Dg^{U}_\Theta(\mathfrak A):=
\{\sigma^U:\sigma\in\Dg_\Theta(\mathfrak A)\}
$ 
and 
$\Th^U(\mathfrak A):=\{\sigma^U:\mathfrak A\vDash\sigma\}$
are useful when the theory of $\mathfrak A$ is considered 
inside a~larger theory.

\begin{thm}[Semantic Completeness]\label{t: sem complet gen}
Let $\mathfrak A$ be a~model pointwise definable 
by formulas in $\Theta=\{\theta_a:a\in A\}$, 
$\T$~a~theory, and $U$~a~unary predicate symbol.
Assume that 
$\T\vdash\forall x\,(\theta_a(x)\to U(x))$ 
for all $a\in A$, and 
$$
\Dg^{U}_\Theta(\mathfrak A)\subseteq
\{\sigma^U:\T\vdash\sigma\}\subseteq\Th^U(\mathfrak A).
$$ 
Then 
$$
\{\sigma^U:\T\vdash^{U}_\Theta\sigma\}=\Th^U(\mathfrak A).
$$
\end{thm}

\hide{

\begin{thm}[Semantic Completeness]\label{t: sem complet gen}
Let $\mathfrak A$ be a~model, 
$\Theta=\{\theta_a:a\in A\}$ where $\theta_a$~is 
a~uniqueness formula defining~$a$, $\T$~a~theory, 
and $U$~a~unary predicate symbol.
Assume that 
\begin{enumerate}
\item[(i)]
$\T\vdash\forall x\,(\theta_a(x)\to U(x))$ 
for all $a\in A$,
\item[(ii)]
$\{\sigma^U:\T\vdash\sigma\}\subseteq\Th(\mathfrak A)$, 
and
\item[(iii)]
for all $\Theta$-atomic~$\sigma$,
$$
\T\vdash\sigma 
\;\text{ iff }\;
\mathfrak A\vDash\sigma. 
$$ 
\end{enumerate}
Then for all sentences~$\sigma$,
$$
\T\vdash^{U}_\Theta\sigma
\;\text{ iff }\;
\mathfrak A\vDash\sigma,
$$
i.e., 
$\{\sigma^U:\T\vdash^{U}_\Theta\sigma\}=\Th(\mathfrak A)$.
\end{thm}

}

\begin{proof} 
Modify the proof of Theorem~\ref{t: sem complet}. 
\end{proof}

Briefly speaking, if $(\T,\vdash)$ correctly thinks 
about $\mathfrak A$ and deduces its generalized diagram, 
then $(\T,\vdash^{U}_\Theta)$ deduces the complete 
theory of~$\mathfrak A$ relativized to~$U$.

In particular, we have the non-relativized variant:

\begin{coro}\label{c: sem complet gen}
Let $\mathfrak A$ be a~model pointwise definable 
by formulas in $\Theta$ and $\T$~a~theory such that 
$$
\Dg_\Theta(\mathfrak A)\subseteq
[\T,\vdash]\subseteq\Th(\mathfrak A).
$$
Then 
$[\T,\vdash_\Theta]=\Th(\mathfrak A).$
\end{coro}

\begin{proof}
Theorem~\ref{t: sem complet gen} for the universal~$U$. 
\end{proof}

\begin{rmk}
Theorem~\ref{t: sem complet gen} can be generalized 
further by using various predicates~$U_i$ to deduce 
complete theories of many models (each relativized 
to its own~$U_i$) simultaneously. Since such 
a~generalization is routine, we leave it to the reader. 
\end{rmk}


\hide{

\newpage

CORRECT!!

[Below $\vdash^{\omega}$ denotes (in the context of $\ZFC$) 
the provability expanded by the relativized $\omega$-rule, 
i.e., $\vdash^{U}_\Theta$ where $U(x)$ means that $x$ is 
a~finite ordinal and $\Theta$~consists of uniqueness 
formulas describing finite ordinals.]

\begin{thm}\label{t: zfc + rel omega-rule}
Let $\theta(x)$ be a~uniqueness formula in $(\ZFC,\vdash)$.
Then $[\ZFC,\vdash^{\omega}]$ contains the complete theory 
of $(x,\in)$ with the set~$x$ satisfying $\theta(x)$, 
in the sense that 
$$
\ZFC\,\vdash^{\omega}\,
\exists x\,(\theta(x)\wedge\sigma^x)
$$
for all $\sigma\in\Th(x,\in)$.
\end{thm}

\begin{proof}
As relativizations of arbitrary formulas to sets are 
$\Delta_1$-formulas, $\ZFC$~proves the L\"owenheim--Skolem 
theorem for sets, in the sense that, 
for all sentences~$\sigma$,
$$
\ZFC\vdash\forall x\exists y\,
(|y|=\omega\wedge\psi(x,y)\wedge(\sigma^x\;\lra\;\sigma^y))
$$
where $\psi(x,y)$ defines a~unique~$y$ from~$x$.

Therefore, we get (in $\ZFC$) that 
$(x,\in)$ is elementarily equivalent to $(\omega,E)$ 
for some $E\subseteq\omega^2$. Now apply to the model 
$(\omega,E)$ the relativized $\omega$-rule, 
or more precisely, 
Theorem~\ref{t: sem complet gen} for this model.
\end{proof}

Note that 
$\{\theta(x)\wedge\sigma^x:(x,\in)\vDash\sigma\}$ 
is a~complete $1$-type; thus the formula $\theta(x)$ 
locally realizes this type in the theory 
$[\ZFC,\vdash^{\omega}]$.

For instance, $[\ZFC,\vdash^{\omega}]$ contains 
the complete theory of $L_\alpha$ with any definable 
(without parameters)~$\alpha$, e.g., for $\omega^{\CK}_1$ 
or $\omega_1$ or the least $\alpha=\beth_\alpha$.

\vskip+1em
[Questions (auxiliary): 

1. 
Why $[\ZFC,\vdash^{\omega}]$ is not complete? 
What it says about CH, does this depend on a~particular~$V$?
Is it true that the existence of inaccessibles is still 
open in it?

2. 
Obviously, $\ZFC+\TA\subseteq[\ZFC,\vdash^{\omega}]$,
show $\ne$ by giving a~witnessing sentence. 

3. 
Can we replace $\ZFC$ by a~weaker theory, e.g., $\Z$ or $\KP$, 
i.e., does it prove the L\"owenheim--Skolem theorem for sets? 

4. 
Generalizations: 
(a) to extensions of $\ZFC$;
(b) to models in a~(countable) language interpreting 
their theories in set theory;
(c) to the language of set theory expanded by, say, 
$\kappa$~constants encoded in $H_\omega(\kappa)$ 
(sets that are hereditary finite over~$\kappa$).

\vskip+1em

}


\subsection*{An example in set theory.}
Let us discuss also an easy ``application'', or rather,
an instance of Theorem~\ref{t: sem complet gen}. 

Let $\alpha$ be the least ordinal such that 
$L_\alpha\prec L$ (or, equivalently, 
$L_\alpha\prec L_{\omega_1}$), 
so $\Th(L_\alpha)=\Th(L)$.
Notice that $L_\alpha$ is pointwise definable (indeed, 
the definable hull of the empty set in $L_\alpha$ 
collapses onto $L_\beta\prec L_\alpha$ for some 
$\beta\le\alpha$ by the G\"odel condensation lemma, 
but then $\beta=\alpha$ by minimality of~$\alpha$).
So let $\theta_a(x)$ define $a\in L_\alpha$, 
and let $\Theta=\{\theta_a:a\in L_\alpha\}$.

Expand the language of $\ZFC$ by a~new 
symbol~$U$ for the set~$L_\alpha$, and 
let $\ZFC^+$ be the extension of $\ZFC$ by the schema 
consisting of the formulas
$\forall x\,(\theta_a(x)\to x\in U)$
for all $a\in L_\alpha$,
and by $\Dg^{U}_\Theta(L_\alpha)$,  
the generalized diagram of $(L_\alpha,\in)$. 
As we have 
$
\{\sigma^{U}:\ZFC^+\vdash\sigma\}
\subseteq\Th(L_\alpha),
$ 
we may apply Theorem~\ref{t: sem complet gen} 
to obtain the complete theory of~$L$:

\begin{prop}\label{p: rel thy of L} 
In the notation above, 
$
\{\sigma^{U}:\ZFC^+\vdash^{U}_{\Theta}\sigma\}=\Th(L).
$
\end{prop}

Note that $\Dg^{U}_\Theta(L_\alpha)$ is not deducible 
in $\ZFC$, or any its recursively axiomatized extension,  
unless $\alpha<\omega^{\CK}_1$. Note also that the set 
(of G\"odel's codes of formulas in) $\Th(L)$ is not 
constructible (by Tarski's indefinability theorem).

\begin{q}
Let $\alpha$ be the least $\alpha$ with $L_\alpha\prec L$. 
Is the set $\Dg^{U}_\Theta(L_\alpha)$ less complex than 
$\Th(L)$, in a~sense? E.g., is it constructible? 
If not, what about $\Dg^{U}_\Theta(L_{\omega^{\CK}_1})$? 
\end{q}

\begin{rmk}
Let us make note of two more related things here.

1. 
The existence of $\alpha$ such that $L_\alpha\prec L$, 
and so, of the least such~$\alpha$, 
is unprovable in $\ZFC$ (by the second G\"odel 
incompleteness theorem) but can be easily seen from 
an external point of view (by using the reflection 
theorem schema). By the same reason, the existence of 
the least $\alpha$ such that $L_\alpha\vDash\ZFC$ 
is unprovable in $\ZFC$ but is true externally. 
As stated in \cite{Myhill}, \cite{Shepherdson}, 
the latter model $L_\alpha$~is pointwise definable 
and the least transitive model of~$\ZFC$; a~stronger 
version of constructibility proposed in \cite{Cohen 1963} 
gives exactly this model. 
\hide{CHECK:
The least $\alpha$ such that 
$L_\alpha\vDash\ZFC$ does not exceed the least $\alpha$ 
such that $L_\alpha\prec_{\Sigma_1} L$. 
, and the latter 
set~$L_\alpha$ consists of all sets that are 
$\Sigma_1$-definable in $L$ without parameters, 
see \cite{Barwise 1975}, V, Corollary~7.9(i). 
}

2. 
A~similar argument works for every model 
of set theory that satisfies $V=\HOD$. Moreover
(see \cite{Enayat 2005}, Theorem~2.11 and 
\cite{Hamkins etal 2013}, Observation~5 onwards),
for any $N\vDash\ZF$ we have $\pd(N)\prec\HOD^N$ (since 
$\pd(N)$ is closed under the canonical Skolem functions), 
and for any model~$M$, it is equivalent to be:
\begin{enumerate}
\item[(i)]
a~pointwise definable model of $\ZFC$,
\item[(ii)]
the pointwise definable part of 
a~model of $\ZFC+V=\HOD$,
\item[(iii)]
a~prime model of $\ZFC+V=\HOD$.  
\end{enumerate}
For more on pointwise definable models of set 
theory and close things, we refer the reader to 
\cite{Paris,Enayat 2004,Hamkins etal 2013,Williams}; 
e.g., \cite{Hamkins etal 2013} shows that every 
countable model of $\ZFC$ and indeed of $\NBG$ 
has a~pointwise definable extension. 
Recall also that $V=\HOD$ is consistent with essentially 
all known large cardinals, 
in contrast with the situation for $L$ and, moreover, 
for core models, which have not yet been constructed 
that can have, e.g., supercompact cardinals.
\end{rmk}


\hide{

+++OBSOLETE

We provide an easy corollary of the semantic 
completeness and consider an its application 
(in fact, we use here the completeness result in 
its general, term-free form as the application 
relates to the language without function symbols).

For every model~$\mathfrak A$, the set of its points
definable over $\mathfrak A$ without parameters, if 
nonempty, forms a~submodel (since closed under operations 
of $\mathfrak A$ if they exist). We call this submodel 
the {\em pointwise definable part} of~$\mathfrak A$. 
In general, this part may have or not have the same 
theory as~$\mathfrak A$, and be or not be definable
over~$\mathfrak A$.

\begin{thm}\label{t: pointwise def part}
Let $\T$ be a~theory and  
$
\Theta=\{\theta:\T\vdash
(\theta\text{ is a~uniqueness formula})\}.
$
Assume that the pointwise definable part of any 
model of $\T$ satisfies~$\T$, and that for every 
atomic $\varphi(x_0,\ldots,x_n)$ and 
$\theta_0,\ldots,\theta_n$ in $\Theta$,
$$
\T\vdash\bar\varphi
\;\text{ or }\;
\T\vdash\neg\,\bar\varphi
$$
whenever $\bar\varphi$ is the sentence
$
\exists x_0\ldots\exists x_n\,
(\varphi(x_0,\ldots,x_n)\wedge 
\theta_0(x_0)\wedge\ldots\wedge\theta_n(x_n)).
$
Then $[\T,\vdash_\Theta]$ is the complete theory 
of the pointwise definable part of any model of~$\T$ 
(and thus all these parts are elementarily equivalent).
\end{thm}

\begin{proof}
Pick any $\bar\varphi$ of the described form and 
any model of $\T$ with $\mathfrak A$ its pointwise 
definable part. If $\T\vdash\bar\varphi$ 
then $\mathfrak A\vDash\bar\varphi$ since 
$\mathfrak A\vDash\T$ by our assumption. 
If $\T\not\vdash\bar\varphi$ then 
$\T\vdash\neg\,\bar\varphi$ by another our assumption, 
then $\mathfrak A\vDash\neg\,\bar\varphi$ 
again since $\mathfrak A\vDash\T$, and thus
$\mathfrak A\not\vDash\bar\varphi$; so in 
contraposition: if $\mathfrak A\vDash\bar\varphi$
then $\T\vdash\bar\varphi$. Therefore, we have 
$\T\vdash\bar\varphi$ iff $\mathfrak A\vDash\bar\varphi$. 
Now apply (the general version of) 
Theorem~\ref{t: sem complet} to complete the proof. 
\end{proof}


We use this to obtain a~completeness result 
in set theory.

\begin{thm}\label{t: ZFC}
Let 
$
\Theta=\{\varphi:\ZFC\vdash
(\varphi\text{ is a~uniqueness formula})\}.
$
Then 
$$[\ZFC,\vdash_\Theta]=\Th(L).$$
\end{thm}

\begin{proof}
Let $\alpha$ be the least ordinal such that 
$L_\alpha\vDash\ZFC$ (the existence of such an $\alpha$ 
is unprovable in $\ZFC$ but can be easily shown by 
using the reflection theorem schema). Known facts (stated 
in \cite{Myhill}, \cite{Shepherdson},~\cite{Cohen 1963}) 
are that the model $L_\alpha$~is pointwise definable,
the least transitive model of~$\ZFC$ and the least 
elementary submodel of~$L$. The set $\Theta$ consists 
of the formulas defining elements of~$L_\alpha$. 
Applying Theorem~\ref{t: pointwise def part}, we get 
that $[\ZFC,\vdash_\Theta]$ is the complete theory 
of $L_\alpha$ and thus of~$L$.
\end{proof}

\hide{
[Hamkins notes that such an~$\alpha$ is $\Delta^{1}_{1}$, 
in the sense that so is a~set real encoded an ordering 
of~$\omega$ having the order-type~$\alpha$; see
https://mathoverflow.net/questions/62708/height-of-minimal-model-of-zfc ]
}

The same remains true if we weaken $\ZFC$ to $\ZF^-$, 
the theory without the axioms of choice, regularity, 
and powerset. This is because the definition of $L$ does 
not use these axioms, so $\ZF^-$ has the same least model.

On the other hand, if we let $\T=\ZFC+(V\ne HOD)$, 
the theory $[\T,\vdash_\Theta]$ becomes inconsistent. 
Indeed, otherwise it should be the complete theory 
of a~pointwise definable model of $\ZFC$. However, 
each such a~model satisfies $V=HOD$ (see, e.g., 
\cite{Hamkins etal 2013}, Observation~5). 
For more on pointwise definable and least models 
of set theory, we refer the reader to \cite{Paris}, 
\cite{Enayat 2004},~\cite{Hamkins etal 2013}.

+++
}


\subsection*{Modal logics of restriction rules}

The $\omega$-rule leads to a~system of provability 
logic, the Japaridze polymodal logic $\GLP$. It is 
natural to expand this construction to the general 
case under our consideration and ask about 
the resulting modal logics.

To distinguish between the syntax of modal 
(propositional) logic and predicate logic, 
we use an upright font for the former, with 
$\mathsf p,\mathsf q,\dots$ to denote propositional 
variables and $\upvarphi,\uppsi,\dots$ modal formulas.
Given an ordinal~$\delta$, the language of $\GLP_\delta$
expands that of the language of classical propositional 
logic by including the necessity operators $[\alpha]$ 
for all $\alpha<\delta$; the dual possibility operators 
$\langle\alpha\rangle$ abbreviate $\neg\,[\alpha]\neg$ 
of course.

The axioms of $\GLP_\delta$ are all classical 
tautologies and all modal formulas of one of 
the following forms: 
\begin{enumerate}
\item[(i)] 
$
[\alpha](\upvarphi\to\uppsi)\to
([\alpha]\upvarphi\to[\alpha]\uppsi),
$
\item[(ii)] 
$
[\alpha]([\alpha]\upvarphi\to\upvarphi)\to 
[\alpha]\upvarphi,
$
\item[(iii)] 
$
[\alpha]\upvarphi\to[\beta]\upvarphi
$
if $\alpha<\beta<\delta$,
\item[(iv)] 
$
\langle\alpha\rangle\upvarphi\to
[\beta]\langle\alpha\rangle\upvarphi
$
if $\alpha<\beta<\delta$;
\end{enumerate}
the rules of inference are modus ponens
and $[0]$-necessitation (deriving 
$[0]\upvarphi$ from~$\upvarphi$).

Let $\T$ be a~theory, $\vdash$ a~provability predicate, 
$\mathfrak A$ a~model pointwise definable by
$\Theta=\{\theta_a:a\in A\}$, and $U$~a~unary 
predicate symbol. Assume that 
$\T\vdash\forall x\,(\theta_a(x)\to U(x))$ 
for all $a\in A$, and 
$
\Dg^{U}_\Theta(\mathfrak A)\subseteq
\{\sigma^U:\T\vdash\sigma\}\subseteq\Th^U(\mathfrak A).
$ 
By Theorem~\ref{t: sem complet gen}, 
every sentence that is true in $\mathfrak A$ is 
of form $\sigma^U$ for some $\sigma$ derived by 
$\vdash^{U}_\Theta$ from~$\T$. 
Define an increasing $\delta$-sequence of theories 
(in the same vocabulary) where each theory indexed 
by a~successor ordinal is obtained from its predecessor  
by a~single application of the $(\Theta,U)$-rule, and 
each theory indexed by a~limit ordinal is the union of 
all theories with smaller indices (as pointed out above,  
such sequences first appeared in~\cite{Rosser 1937}): 
\begin{align*}
\T_0&:=\T,
\\
\T_{\alpha+1}&:=\T_\alpha\cup
\bigl\{(\forall x\,\varphi(x))^U:
\{(\exists x\,(\varphi(x)\wedge\theta(x)))^U:
\theta\in\Theta\}
\subseteq[\T_{\alpha},\vdash] 
\bigr\},
\\
\T_{\alpha}&:=\bigcup_{\beta<\alpha}\T_{\beta} 
\text{ if $\alpha$ is a~limit ordinal}.
\end{align*}

Assume that $\T$ is strong enough to encode 
its syntax and the provability predicate (as, 
e.g., $\PA$ or $\ZFC$), and let $\Pr_\alpha(x)$ 
encode ``$x$~is the code of a~sentence provable 
in~$\T_\alpha$'', for each $\alpha<\delta$. 
Then we define an~{\em interpretation} 
of the language of $\GLP_\delta$ in the language 
of~$\T$ as any function~${}^*$ that sends each 
propositional variable~$\mathsf p$ of the language 
of $\GLP_\delta$ to a~sentence $\mathsf p^*$ of 
the language of~$\T$, and then extends to all 
modal formulas 
by stipulating that ${}^*$ commutes with 
the Boolean connectives, and that 
$([\alpha]\upvarphi)^*$ is 
$\Pr_\alpha(\ulc\upvarphi^*\urc)$, where 
$\ulc\psi\urc$ stands for the code of~$\psi$.

For the $\omega$-rule, i.e., if $\mathfrak A$ is the 
standard model of arithmetic and $U$~the universal 
relation, 
it has been shown by Japaridze~\cite{Japaridze} 
that $\GLP_\omega$ is sound and complete for~$\PA$, 
in the sense that a~formula $\upvarphi$ is provable 
in $\GLP_\omega$ iff, for every interpretation~${}^*$, 
the sentence $\upvarphi^*$ is provable in~$\PA$; 
see also~\cite{Beklemishev}. Boolos~\cite{Boolos} 
obtained an analogous result for $\GLP_2$
and second-order arithmetic. 
Ignatiev~\cite{Ignatiev} provided general conditions 
on abstract provability-like predicates~$\Pr_n$ under 
which $\GLP_\omega$ is sound and complete for~$\PA$. 
For recent studies of $\GLP$ see also 
\cite{Beklemishev,Beklemishev Gabelaia}.
Let us also mention that the approach developed in  
\cite{Saveliev Shapirovsky 2019} allows one to obtain 
proof-theoretic logics like $\GLP$ and model-theoretic
logics in a~unified way 
(see \cite{Saveliev Shapirovsky 2019}, Example~1).

\begin{q}
For which theories~$\T$, pointwise definable 
models $\mathfrak A$ of~$\T$, and $(\Theta,U)$-rules, 
are the modal logics $\GLP_\delta$ sound and complete? 
In general, what ($\GLP_\delta$-like) logics arises 
from various such $\T$, $\mathfrak A$, and restriction 
rules? 
\end{q}


\hide{
\subsection*{Realization rules}

[or: inflation rules? saturation rules?]

Looking in the opposite direction, we can wonder about 
rules leading rather to saturated than to prime models. 
Here we only briefly mention such a~possibility.

\vskip+1em
[Note that $\ZFC$ proves that a~saturated model of $\ZF$ 
of cardinality~$\kappa$ exists iff $\ZF$~is consistent 
and $\kappa^{<\kappa}=\kappa>\omega$, see...? 
It seems, moreover, the statement that every model 
has a~saturated elementary extension is equivalent to 
the existence of a~proper class of cardinals~$\kappa$ 
such that $\kappa^{<\kappa}=\kappa>\omega$ 
(and hence, it is not provable in ZFC since 
$\kappa^{<\kappa}=\kappa>\omega$ clearly means that either 
$\kappa$~is inaccessible or $\kappa=\lambda^+=2^\lambda$
for some~$\lambda$). Reference??]
\vskip+1em

Given a~regular cardinal~$\kappa$ and a~vocabulary~$\tau_0$, 
let $C_{\alpha+1}$ consist of constant symbols~$c_\Phi$ 
for each 
$\Phi\in\mathscr P_\kappa(\mathscr L(\tau_\alpha))$, and 
let $\tau_{\alpha+1}:=\tau_\alpha\cup C_{\alpha+1}$ and 
$\tau_\alpha:=\bigcup_{\beta<\alpha}\tau_\alpha$ 
if $\alpha$ is a~limit ordinal. 
The $\kappa$-{\em realization rule} is the schema 
consisting of the rules
\hide{
$$
\frac
{
\bigl\{\exists x\bigwedge\Phi_0(x):
\Phi_0\in\mathscr P_\omega(\Phi(x))\bigr\}
}
{\Phi(c_\Phi)}.
$$ 
}
$$
\frac
{
\bigl\{
\exists x\,
(\varphi_0(x)\wedge\ldots\wedge\varphi_{n-1}(x)):
\varphi_i\in\Phi(x),i<n<\omega
\bigr\}
}
{\Phi(c_\Phi)}
$$ 
for each 
$\Phi\in\mathscr P_\kappa(\mathscr L(\tau_\kappa))$.  
[Does every model of a~theory closed under this rule 
be $\kappa$-saturated? check up!]

}


\section{Restriction rules, higher logics}

In this section, we discuss generalizations of 
restriction rules to higher logics. Mainly we consider 
infinitary logics, and only briefly second-order logic.

\subsection*{Infinitary logics}
Given cardinals $\kappa\ge\lambda$, the first-order 
infinitary logic $\mathscr L_{\kappa,\lambda}$ involves 
propositional connectives of arity~${<}\kappa$ and 
quantifiers bounding~${<}\lambda$ variables. 
As will soon be seen, these logics admit a~rather 
straightforward generalization of our previous observations.

\hide{
Checking the proof of Theorem~\ref{t: syn complet}, 
we may note that in order to generalize it to 
$\mathscr L_{\kappa,\lambda}$ we need, besides 
suitable restriction rules, only the following primitive 
list of infinitary axioms and rules:
\begin{enumerate}
\item[(i)] 
the {\em conjunction axiom} $\bigwedge\Phi\to\varphi$
for $|\Phi|<\kappa$ and $\varphi\in\Phi$,  
\item[(ii)] 
the {\em conjunction rule} ${\Phi}/{\bigwedge\Phi}$
for $|\Phi|<\kappa$, and 
\item[(iii)] 
the {\em generalization rule} 
${\varphi}/\,{\bigforall X\varphi}$
for $|X|<\lambda$. 
\end{enumerate}
(Such a~deductive system is called {\em basic}, it does not 
include more complex axioms and rules related to the axiom 
of choice; see \cite{Karp 1964} and Appendix~C in 
\cite{Dickmann 1975} for more on this.)
}


Let us recall common Hilbert-style deductive systems 
for the language $\mathscr L_{\kappa,\lambda}$. 
(In the formulas below, the indices of any connective are 
assumed to be bounded by a~positive ordinal~$<\kappa$, 
and of any quantifier, by an ordinal~$<\lambda$.)
First consider the system in which axioms consist of 
five propositional axiom schemas
\begin{enumerate}
\item[(i)] 
$
\varphi\to(\psi\to\varphi),
$
\item[(ii)] 
$
(\varphi\to(\psi\to\chi))\to
((\varphi\to\psi)\to(\varphi\to\chi)),
$
\item[(iii)] 
$
(\neg\,\varphi\to\neg\,\psi)\to(\psi\to\varphi), 
$
\item[(iv)] 
$
\bigwedge_{\alpha<\gamma}(\varphi\to\varphi_\alpha)\to 
(\varphi\to\bigwedge_{\alpha<\gamma}\varphi_\alpha),
$
\item[(v)] 
$ 
\bigwedge_{\alpha<\gamma}\varphi_\alpha
\to\varphi_\beta,
$ 
\end{enumerate}
two quantifier axiom schemas
\begin{enumerate}
\item[(i)]
$
\bigforall_{\beta<\delta}x_\beta\,
(\varphi\to\psi) 
\to
\bigl(
\varphi\to\bigforall_{\beta<\delta}x_\beta\,\psi
\bigr), 
$
provided that $\varphi$ has 
no free occurrence of $x_\beta$, 
\item[(ii)] 
$
\bigforall_{\beta<\delta}x_\beta\,
\varphi(x_\beta)_{\beta<\delta}
\to
\varphi(t_\beta)_{\beta<\delta}
$
provided that $\varphi$ has 
no free occurrence of $x_\beta$ 
within the scope of a~quantifier 
binding some variable of~$t_\beta$,  
\end{enumerate}
and two equality axiom schemas
\begin{enumerate}
\item[(i)]
$t=t$ for all terms~$t$,
\item[(ii)] 
$
\bigwedge_{\alpha<\gamma}(s_\alpha=t_\alpha)\to
(
\varphi(s_\alpha)_{\alpha<\gamma}
\lra
\varphi(t_\alpha)_{\alpha<\gamma}
),
$
provided that no variables in all terms 
$s_\alpha,t_\alpha$ is bound in~$\varphi$,
\end{enumerate}
while rules consist of two propositional rule schemas
\begin{enumerate}
\item[(i)] 
{\em Modus ponens}
$
\displaystyle
\frac{\{\varphi,\varphi\to\psi\}}{\psi}, 
$
\item[(ii)]
{\em Conjunction}
$
\displaystyle
\frac{\{\varphi_\alpha:\alpha<\gamma\}}
{\bigwedge_{\alpha<\gamma}\varphi_\alpha}, 
$
\end{enumerate}
and one quantifier rule schema
%
\begin{enumerate}
\item[] 
{\em Generalization}
$
\displaystyle
\frac{\varphi}
{\;
\bigforall_{\alpha<\gamma}x_\alpha\,\varphi
\,}, 
$
\end{enumerate}
The deductive system formed by these axioms and rules is 
called {\em basic} (for $\mathscr L_{\kappa,\lambda}$).

Consider also the following additional axioms and rules:
Chang's distributivity axiom schema 
\begin{enumerate}
\item[(D$_\gamma$)]
{\em Distributivity}
$
\bigvee_{\alpha<\gamma}
\bigwedge_{\beta<\gamma}
\varphi_{\alpha,\beta},
$
where for any $(\varphi_\delta)_{\delta<\gamma}$  
there is $\varepsilon<\gamma$ with 
$
\varphi_{\alpha,\beta}\in 
\{\varphi_\varepsilon,\neg\,\varphi_\varepsilon\},
$ 
and for any $f:\gamma\to\gamma$ 
there is $\varepsilon<\gamma$ with 
$
\{\varphi_\varepsilon,\neg\,\varphi_\varepsilon\}
\subseteq
\{\varphi_{\alpha,f(\alpha)}:\alpha<\gamma\},
$
\end{enumerate}
and two choice rule schemas
\begin{enumerate} 
\item[(C$_\kappa$)] 
{\em Independent choice}
$
\displaystyle
\frac{
\bigvee_{\alpha<\gamma}\varphi_\alpha
}
{\:
\bigvee_{\alpha<\gamma}
\bigforall_{\beta<\delta_\alpha}x_{\alpha,\beta}\,
\varphi_\alpha
\,}, 
$
provided that all the variables $x_{\alpha,\beta}$ 
are pairwise distinct and no $x_{\alpha,\beta}$ 
occurs freely in $\varphi_\varepsilon$ unless 
$\alpha=\varepsilon$.
\item[(DC$_\kappa$)] 
{\em Dependent choice} 
$
\displaystyle
\frac{
\bigvee_{\alpha<\gamma}\varphi_\alpha
}
{\:
\bigforall_{\beta<\delta_0}\!x_{0,\beta}\,
\varphi_0
\vee
\bigvee_{0<\alpha<\gamma}
\bigexists_{\beta<\delta_\alpha}\!y_{\alpha,\beta}
\bigforall_{\beta<\delta'_\alpha}\!x_{\alpha,\beta'}\,
\varphi_\alpha
\,}, 
$
provided that all the variables $x_{\alpha',\beta'}$ 
and $y_{\alpha,\beta}$ are pairwise distinct, 
no $x_{\alpha,\beta'}$ occurs freely in 
$\varphi_\varepsilon$ if $\varepsilon<\alpha$, and 
all free variables of $\varphi_\alpha$, as well as 
all $x_{\varepsilon,\beta'}$ with $\varepsilon<\alpha$, 
are in $\{y_{\alpha,\beta}:\beta<\delta_\alpha\}$.
\end{enumerate}

\hide{
For brevity, we denote this Chang's distributivity 
schema by $\mathrm{D}_\kappa$, and these two choice 
schemas, by $\mathrm{C}_\kappa$ and $\mathrm{DC}_\kappa$, 
respectively. 
}
Let us say that the deductive system for 
$\mathscr L_{\kappa,\lambda}$ is {\em weak standard}, 
respectively, {\em standard} iff it consists of 
the basic system augmented by $\mathrm{D}_\gamma$ and by 
$\mathrm{C}_\gamma$, respectively, $\mathrm{DC}_\gamma$, 
for all $\gamma<\kappa$.


Assuming that a~deductive system of a~given language 
$\mathscr L$ is consistent, let us call it {\em complete} 
iff the set of deducible formulas of $\mathscr L$ 
coincides with the set of its formulas that are valid, 
i.e., satisfied on all models of~$\mathscr L$,  
symbolically, $\vdash\varphi$ iff $\vDash\varphi$, 
and {\em strongly complete} iff for any theory~$\T$ 
of~$\mathscr L$, the set of formulas deducible in $\T$ 
coincides with the set of formulas valid in all models 
of~$\T$, i.e., 
$\T\vdash\varphi$ iff $\T\vDash\varphi$.

\hide{
The following and some other completeness results are 
due to Karp, who also established that the principles 
added to basic system cannot be weakened without loss 
of completeness (see~\cite{Karp 1964}):
The basic deductive system of $\mathscr L_{\omega_1,\omega}$
is complete;
the weak standard system of $\mathscr L_{\kappa,\omega}$ 
is complete whenever $\kappa$ is inaccessible;
the standard system of $\mathscr L_{\kappa^+,\lambda}$ 
is complete whenever $\kappa^{<\lambda}=\kappa$;
the standard system of $\mathscr L_{\kappa,\lambda}$ 
is complete whenever $\kappa$ is weakly inaccessible 
and $\gamma^\delta<\kappa$ for all $\gamma<\kappa$ and 
$\delta<\lambda$, in particular, whenever $\kappa$ is 
inaccessible.
}

The following and some other completeness results are 
due to Karp, who also established that the principles 
added to basic system cannot be weakened without loss 
of completeness (the reader may consult \cite{Karp 1964} 
for details or \cite{Dickmann 1975}, App.~C,~I
for a~short overview):
\begin{enumerate}
\item[(i)] 
the basic deductive system of 
$\mathscr L_{\omega_1,\omega}$ is complete;
\item[(ii)] 
the weak standard system of $\mathscr L_{\kappa,\omega}$ 
is complete if $\kappa$ is inaccessible;
\item[(iii)] 
the standard system of $\mathscr L_{\kappa^+,\lambda}$ 
is complete if $\kappa^{<\lambda}=\kappa$;
\item[(iv)] 
the standard system of $\mathscr L_{\kappa,\lambda}$ 
is complete if $\kappa$ is weakly inaccessible 
and $\gamma^\delta<\kappa$ for all $\gamma<\kappa$ and 
$\delta<\lambda$, in particular, if $\kappa$ is 
inaccessible.
\end{enumerate}

This implies that in the cases of (i)--(iv),  
the set $\Val(\mathscr L_{\kappa,\lambda}(\in))$ 
of (codes of) formulas of 
$\mathscr L_{\kappa,\lambda}(\in)$ 
that are valid, i.e., true in all models of~$\{\in\}$, 
is $\mathscr L_{\omega,\omega}(\in)$-definable 
over the coding structure $(H_\kappa,\in)$.  
More precisely, it is $\Sigma_2$-definable, and 
if $\kappa$~is inaccessible, even  
$\Sigma_1$-definable over $(H_\kappa,\in)$; 
here $\Sigma_n$ refers to L{\'e}vy's hierarchy 
of set-theoretical formulas.

Recall that a~theory~$\T$ is $\kappa$-{\em consistent} 
iff any subtheory $\T_0$ in $\mathscr P_\kappa(\T)$ is 
consistent, and that $\kappa$~is {\em strongly compact} 
iff for any vocabulary~$\tau$ and theory $\T$ of 
$\mathscr L_{\kappa,\kappa}(\tau)$, if $\T$ is 
$\kappa$-consistent then it is consistent. 
Immediately from this definition and items (ii) and~(iv) 
above we obtain examples of strongly complete logics:%
\footnote{
Despite the following weird claim in \cite{Dickmann 1975}, 
p.~419: ``This last remark [saying that 
$\mathscr L_{\omega_1,\omega}$ does not admit 
a~strongly complete deductive system] also applies 
to any language $\mathscr L_{\kappa,\lambda}$''.
}
\begin{enumerate}
\item[(v)] 
the weak standard system of $\mathscr L_{\kappa,\omega}$ 
is strongly complete if $\kappa$ is strongly compact,
\item[(vi)] 
the standard system of $\mathscr L_{\kappa,\lambda}$ 
is strongly complete if $\kappa$ is strongly compact.
\end{enumerate}

In contrast to these completeness results, 
Scott's undefinability theorem states that 
the set $\Val(\mathscr L_{\kappa^+,\kappa^+}(\in))$ 
is not definable over $(H_{\kappa^+},\in)$
even by any formula of the language 
$\mathscr L_{\kappa^+,\kappa^+}(\in)$ itself, 
thus showing the language does not admit 
any reasonably definable complete deductive system 
(see \cite{Karp 1964} or 
\cite{Dickmann 1975}, App.~C,~II).


In order to formulate variants of restriction rules 
that would be appropriate for languages 
$\mathscr L_{\kappa,\lambda}$, note first that 
if $\lambda>\omega$, such rules must take into account 
infinitely many parameters in the formulas. 
As a~simple particular case, the (non-relativized) term 
form of restriction rule schema given by a~set~$S$ of 
terms would state:
$$
\frac
{
\bigl\{
\varphi(t_\alpha)_{\alpha<\gamma}:
(t_\alpha)_{\alpha<\gamma}\in S^\gamma
\bigr\}
}
{
\bigforall_{\alpha<\gamma}x_\alpha\,
\varphi(x_\alpha)_{\alpha<\gamma}
},
$$ 
for all formulas~$\varphi$ in an arbitrary number 
$\gamma<\lambda$ of parameters. 
In the general setting, if $U$ is a~unary predicate 
and $\Theta$ a~set of uniqueness formulas in a~system 
$(\T,\vdash)$ under consideration, 
%
the {\em $(\Theta,U)$-restriction rule} is the schema of 
rules that, for each formula~$\varphi$ in free variables 
$(x_\alpha)_{\alpha<\gamma}$, $\gamma<\lambda$, states:
$$
\frac
{\bigl\{
\bigl(
\bigexists_{\alpha<\gamma}x_\alpha\,
\bigl(
\bigwedge_{\alpha<\gamma}\theta_\alpha(x_\alpha)
\wedge\varphi(x_\alpha)_{\alpha<\gamma}
\bigr)
\bigr)^U:
(\theta_\alpha)_{\alpha<\gamma}\in\Theta^\gamma
\bigr\}}
{\bigl(
\bigforall_{\alpha<\gamma}x_\alpha\,
\varphi(x_\alpha)_{\alpha<\gamma}
\bigr)^U}.
$$

Now we are going to generalize the results above, 
Theorems \ref{t: syn complet gen} 
and~\ref{t: sem complet gen}, allowing to obtain, for 
a~given pointwise $\mathscr L_{\kappa,\lambda}$-definable 
model, its complete $\mathscr L_{\kappa,\lambda}$-theory.  
Surprisingly, such a~generalization is very immediate,   
holds for any infinitary language even it is highly 
incomplete like $\mathscr L_{\kappa^+,\kappa^+}$, 
and does not requires nothing but the basic deductive 
system of the language under consideration, to which 
we add the corresponding restriction rule.

For this, first observe the following fact.

\begin{lmm}\label{l: commuting with uniqueness}
Let $\T$ be a~theory and $\vdash$ the basic provability 
of $\mathscr L_{\kappa,\lambda}$, and let 
$\gamma<\kappa$ and $\delta<\lambda$. 
If $(\T,\vdash)$ proves that 
$(\theta_\beta)_{\beta<\delta}$ 
consists of uniqueness formulas, then 
it proves the following equivalences, 
for all formulas $\varphi,\varphi_\alpha$: 
\begin{enumerate}
\item[(i)]
$
\bigexists_{\beta<\delta}\!x_\beta\,
\bigl(
\bigwedge_{\beta<\delta}\theta_\beta(x_\beta)
\wedge
\neg\,
\varphi
\bigr) 
\lra
\neg
\bigexists_{\beta<\delta}\!x_\beta\,
\bigl(
\bigwedge_{\beta<\delta}\theta_\beta(x_\beta)
\wedge
\varphi
\bigr), 
$
\item[(ii)]
$
\bigexists_{\beta<\delta}\!x_\beta\,
\bigl(
\bigwedge_{\beta<\delta}\theta_\beta(x_\beta)
\wedge
\bigwedge_{\alpha<\gamma}
\varphi_\alpha
\bigr) 
\lra
\bigwedge_{\alpha<\gamma}
\bigexists_{\beta<\delta}\!x_\beta\,
\bigl(
\bigwedge_{\beta<\delta}\theta_\beta(x_\beta)
\wedge
\varphi_\alpha
\bigr), 
$
\item[(iii)] 
$
\bigexists_{\beta<\delta}\!x_\beta\,
\bigl(
\bigwedge_{\beta<\delta}\theta_\beta(x_\beta)
\wedge
\bigforall_{\alpha<\gamma}\!y_\alpha\,
\varphi
\bigr) 
\lra
\bigforall_{\alpha<\gamma}\!y_\alpha
\bigexists_{\beta<\delta}\!x_\beta\,
\bigl(
\bigwedge_{\beta<\delta}\theta_\beta(x_\beta)
\wedge
\varphi
\bigr). 
$
\end{enumerate} 
Consequently, for any formula~$\psi$, the part 
$ 
\bigexists_{\beta<\delta}\!x_\beta\,
\bigl(
\bigwedge_{\beta<\delta}\theta_\beta(x_\beta)
\wedge\ldots)
$ 
of the formula $\bar\psi$ can be 
pushed down inside, or brought out it.
\end{lmm}

\begin{proof}
The ($\to$)~parts of each of (i)--(iii) use that 
$
\forall x\forall y\,
(\theta_\beta(x)\wedge\theta_\beta(y)\to x=y)
$
while their ($\leftarrow$)~parts use that 
$\exists x\,\theta_\beta(x)$ 
whenever $\theta_\beta$~is a~uniqueness formula. 
\end{proof}



\begin{thm}[Syntactic Completeness]%
\label{t: syn complet inf}
Let $\kappa\ge\lambda\ge\omega$ be cardinals, 
$\T$ a~theory and $\vdash$ the basic provability 
of $\mathscr L_{\kappa,\lambda}$, 
$\Theta$ a~set of uniqueness formulas in $(\T,\vdash)$, 
and $U$~a~unary predicate symbol. Assume that 
\begin{enumerate}
\item[(i)] 
$\T\vdash\forall x\,(\theta(x)\to U(x))$
for all $\theta\in\Theta$, and 
\item[(ii)] 
for every atomic formula~$\varphi$ in $n+1$~variables 
and $\theta_0,\ldots,\theta_n\in\Theta$,
$$
\T\vdash\bar\varphi
\;\text{ or }\;
\T\vdash\neg\,\bar\varphi
$$
where $\bar\varphi$ is the sentence 
$
\exists x_0\ldots\exists x_n\,
(\varphi(x_0,\ldots,x_n)\wedge
\theta_0(x_0)\wedge\ldots\wedge\theta_n(x_n)).
$
\end{enumerate}
Then for every sentence~$\sigma$,
$$
\T\vdash^{U}_\Theta\sigma^{U}
\;\text{ or }\;\
\T\vdash^{U}_\Theta\neg\,\sigma^{U},
$$
i.e., the theory $\{\sigma^U:\T\vdash^{U}_\Theta\sigma\}$ 
is a~complete theory relativized to~$U$.
\end{thm}

\begin{proof}
Clearly, it suffices to verify that 
for every formula $\psi(x_\beta)_{\beta<\delta}$ 
of $\mathscr L_{\kappa,\lambda}$ 
in $\delta<\lambda$ variables~$x_\beta$, we have  
$$
\T\vdash^{U}_\Theta\bar\psi^{\,U}
\;\text{ or }\;
\T\vdash^{U}_\Theta\neg\,\bar\psi^{\,U} 
$$
whenever $\bar\psi$ is the sentence 
$
\bigexists_{\beta<\delta}x_\beta\,
\bigl(
\bigwedge_{\beta<\delta}\theta_\beta(x_\beta)
\wedge
\psi(x_\beta)_{\beta<\delta}
\bigr)
$
for some $\delta$-sequence 
$(\theta_\beta)_{\beta<\delta}$ in~$\Theta$. 
Let us prove this by induction on construction of~$\psi$.

\begin{atomic}
By assumption (and since $\vdash^{U}_\Theta$ 
includes~$\vdash$).
\end{atomic}
              
\begin{negation}
Let $\psi$ be $\neg\,\varphi$.
Fix some~$\bar\psi$; then by 
Lemma~\ref{l: commuting with uniqueness}(i), 
in $(\T,\vdash)$, $\bar\psi$ is equivalent to 
$\neg\,\bar\varphi$ with the inherited~$\bar\varphi$. 
By induction hypothesis, the claim holds for~$\bar\varphi$, 
and then it trivially holds for $\neg\,\bar\varphi$, 
and so for~$\bar\psi$.
\end{negation}

\begin{conjunction}
Let $\psi$ be $\bigwedge_{\alpha<\gamma}\varphi_\alpha$.
Fix some~$\bar\psi$; then by 
Lemma~\ref{l: commuting with uniqueness}(ii), 
in $(\T,\vdash)$, $\bar\psi$ is equivalent to 
$\bigwedge_{\alpha<\gamma}\bar\varphi_\alpha$ with 
the inherited $\bar\varphi_\alpha$. 
We have two cases:

\begin{case1}
$\T\vdash^{U}_\Theta\bar\varphi_\alpha$ 
for all $\alpha<\gamma$. 
Then 
$
\T\vdash^{U}_\Theta
\bigwedge_{\alpha<\gamma}\bar\varphi_\alpha
$ 
by the conjunction rule, and so 
$\T\vdash^{U}_\Theta\bar\psi$ .
\end{case1}

\begin{case2}
$\T\not\vdash^{U}_\Theta\bar\varphi_\alpha$ 
for some $\alpha<\gamma$. Then 
$\T\vdash^{U}_\Theta\neg\,\bar\varphi_\alpha$ 
(for this~$\alpha$) by inductive hypothesis. 
Therefore, since 
$
\neg\,\bar\varphi_\alpha\to
\neg\bigwedge_{\alpha<\gamma}\bar\varphi_\alpha
$ 
is an instance of the conjunction axiom in contraposition,
$
\T\vdash^{U}_\Theta
\neg\bigwedge_{\alpha<\gamma}\bar\varphi_\alpha
$ 
by modus ponens, and so 
$\T\vdash^{U}_\Theta\neg\,\bar\psi$.
\end{case2}
\end{conjunction}

\begin{universal quantifier}
Let $\psi$ be 
$
\bigforall_{\alpha<\gamma}x_\alpha\,
\varphi(x_\alpha,y_\beta)_{\alpha<\gamma,\beta<\delta}.
$
Fix some~$\bar\psi$; then by 
Lemma~\ref{l: commuting with uniqueness}(iii), 
in $(\T,\vdash)$, $\bar\psi$ is equivalent to 
$
\bigforall_{\alpha<\gamma}x_\alpha\,
\chi(x_\alpha)_{\alpha<\gamma}
$
where $\chi$ is the formula 
$
\bigexists_{\beta<\delta}\!y_\beta\,
\bigl(
\bigwedge_{\beta<\delta}\theta_\beta(y_\beta)
\wedge
\varphi(x_\alpha,y_\beta)_{\alpha<\gamma,\beta<\delta}
\bigr)
$ 
with the inherited~$(\theta_\beta)_{\beta<\delta}$. 
Again we have two cases:

\begin{case1}
$\T\vdash^{U}_\Theta\bar\chi$ for all~$\bar\chi$.
Then 
$
\T\vdash^{U}_\Theta
\bigforall_{\alpha<\gamma}x_\alpha\,
\chi(x_\alpha)_{\alpha<\gamma}
$ 
by the $(\Theta,U)$-rule, 
and so $\T\vdash^{U}_\Theta\bar\psi$ .
\end{case1}

\begin{case2}
$\T\not\vdash^{U}_\Theta\bar\chi$ for some~$\bar\chi$.
Then $\T\vdash^{U}_\Theta\neg\,\bar\chi$ 
(for this~$\bar\chi$) by inductive hypothesis. 
Therefore, 
$
\T\vdash^{U}_\Theta\neg\,
\bigforall_{\alpha<\gamma}x_\alpha\,
\chi(x_\alpha)_{\alpha<\gamma}
$ 
by the generalization rule, 
and so $\T\vdash^{U}_\Theta\neg\,\bar\psi$.
\end{case2}
\end{universal quantifier}

The theorem is proved. 
\end{proof}


Noting that $(\Theta,U)$-rule preserves validity 
in every model in which the submodel definable by $U$ 
is pointwise definable by formulas in~$\Theta$, 
we get a~straightforward generalization of 
Theorem~\ref{t: sem complet gen} 
(or~\ref{t: sem complet}):

\begin{thm}[Semantic Completeness]%
\label{t: sem complet inf}
Let $\mathfrak A$ be a~model pointwise definable 
by formulas of $\mathscr L_{\kappa,\lambda}$ 
in $\Theta=\{\theta_a:a\in A\}$, 
$\T$~a~theory, $\vdash$~the standard system, 
and $U$~a~unary predicate symbol.
Assume that 
\begin{enumerate}
\item[(i)] 
$\T\vdash\forall x\,(\theta_a(x)\to U(x))$ 
for all $a\in A$, and 
\item[(ii)] 
$
\Dg^{U}_\Theta(\mathfrak A)\subseteq
\{\sigma^U:\T\vdash\sigma\}\subseteq\Th^U(\mathfrak A).
$
\end{enumerate}
Then 
$$
\{\sigma^U:\T\vdash^{U}_\Theta\sigma\}=\Th^U(\mathfrak A).
$$
\end{thm}

\begin{proof}
Modify the proof of Theorem~\ref{t: sem complet}.    
\end{proof}

\begin{rmk} 
Although we do not discuss here possible examples or 
applications, 
concerning analogs of the model~$L$ in infinitary and 
second-order languages, let us shortly mention the standard 
result in \cite{Chang}, \cite{Myhill Scott}, and the recent 
article~\cite{Kennedy etal}.
\end{rmk}


Now we consider a~special case of 
$\mathscr L_{\kappa,\lambda}$ in which an analog 
of above mentioned Barwise's result holds 
even in a~stronger form.

Benda proved \cite{Benda 1978} that 
$\kappa$~is supercompact iff the language 
$\mathscr L_{\kappa,\lambda}$ has a~kind of 
``compactness for omitting type'' in the following sense:
for every theory~$\T$ and type $\Gamma(x,y)$ 
in two variables, if each subtype 
$\Gamma_0\in\mathscr P_\kappa(\Gamma)$ has 
a~model of $\T$ that realizes $x$ and omits~$y$, 
then the whole~$\Gamma$ has such a~model.
Using this fact, we can generalize Barwise's 
completeness result to $\mathscr L_{\kappa,\lambda}$ 
with supercompact~$\kappa$.

\hide{
[In fact, the proof in~\cite{Benda 1978} gives 
the characterization of $\mu$-supercompact $\kappa$ 
via types of size~$\le\mu$. 
Since $\kappa$ is $\mu$-supercompact iff 
$\mathscr P_\kappa(\mu)$ carries a~fine normal 
ultrafilter, it trivially is $\mu$-strongly compact 
in the sense that $\mathscr P_\kappa(\mu)$ carries 
a~fine ultrafilter. However, below by ``$\kappa$~is 
$\mu$-strongly compact'' we mean that any 
$\kappa$-satisfiable theory of 
$\mathscr L_{\kappa,\lambda}(\tau)$ is satisfiable 
whenever $|\tau|\le\mu$, which may differ. Check or 
rewrite for the case when $\kappa$ is just supercompact!]

}

\begin{thm}[Strong Semantic Completeness]%
\label{t: Barwise for supercompact}
Let $\lambda\le\kappa\le\mu$, 
$\kappa$ a~$\mu$-supercompact cardinal, 
$\tau$~a~vocabulary containing a~unary predicate 
symbol~$U$, $|\tau|\le\mu$,
$\T$~a~theory of $\mathscr L_{\kappa,\lambda}(\tau)$, 
$\mathfrak A$~a~model of~$\tau$ pointwise definable by 
$\Theta$, and let $\vdash$~be the standard system.  
Assume that 
$\T\vdash\forall x\,(\theta(x)\to U(x))$ 
for all $\theta\in\Theta$, and 
$
\Dg^{U}_\Theta(\mathfrak A)\subseteq
\{\sigma^U:\T\vdash\sigma\}\subseteq\Th^U(\mathfrak A).
$ 
Then for every sentence~$\sigma$ of the language
$\mathscr L_{\kappa,\lambda}(\tau)$ we have:
$$
\T\vdash^{U}_\Theta\sigma
\;\;\text{iff}\;\;\;
\T\vDash^{U}_\Theta\sigma. 
$$
\end{thm}


\begin{proof}
{\em Only if.}
Assume $\T\vdash^{U}_\Theta\sigma$. If 
$\mathfrak B\vDash\T$ and $\mathfrak A$ is 
the submodel of $\mathfrak B$ defined by~$U$, prove 
$\mathfrak B\vDash\sigma$ by induction on complexity 
of~$\sigma$ starting from $\Dg^{U}_\Theta(\mathfrak A)$
by using that the $(\Theta,U)$-rule preserves validity 
in such models~$\mathfrak B$. 

{\em If.}
Assume that for all models $\mathfrak B$ such that 
$\mathfrak B\vDash\T$ and $\mathfrak A$ is 
the submodel of $\mathfrak B$ defined by~$U$, 
we have $\mathfrak B\vDash\sigma$, and prove 
$\T\vdash^{U}_\Theta\sigma$. Thus in contraposition, 
assume that $\neg\,\sigma$ is consistent with 
$(\T,\vdash^{U}_\Theta)$ and prove that 
there exists a~model $\mathfrak B$ such that 
$\mathfrak B\vDash\T+\neg\,\sigma$ in which 
$U$ defines the submodel~$\mathfrak A$, thus 
$\mathfrak B$ omits the $1$-type 
$
\Gamma(x):=
\{U(x)\wedge\neg\,\theta(x):\theta\in\Theta\},
$
in other words, satisfies the sentence 
$
\forall x\,(U(x)\to 
\bigvee_{\theta\in\Theta}\theta(x))
$
of the language $\mathscr L_{\mu,\omega}$).

Let $\Theta_0\in\mathscr P_\kappa(\Theta)$,  
$U_0$~a~unary predicate symbol, and $\T'$ 
the expansion of $\T$ by the formula 
$
\forall x\,(U_0(x)\;\lra\;
\bigvee_{\theta\in\Theta_0}\theta(x)).
$
Then $\neg\,\sigma$ is consistent with 
$(\T',\vdash^{U}_\Theta)$ and a~fortiori with 
$(\T',\vdash)$. Therefore, since $\kappa$~is 
$\mu$-strongly compact, there exists a~model 
of $\T'+\neg\,\sigma$ omitting the type 
$
\Gamma_0(x):=
\{U(x)\wedge\neg\,\theta(x):\theta\in\Theta_0\}.
$
As $\Gamma_0$ is an arbitrary small subtype of~$\Gamma$, 
by the compactness for omitting type we conclude that 
there exists a~model of $\T'+\neg\,\sigma$ omitting 
the whole type $\Gamma$, as required. 
\end{proof}


\subsection*{Second-order logics}

A~second-order formula $\theta(X)$ in 
a~single $(n+1)$-ary relation variable~$X$ is 
a~{\em uniqueness formula} in $(\T,\vdash)$ iff 
$$
\T\;\vdash\;
\forall X\forall Y\,
(\theta(X)\wedge\theta(Y)\to
\forall x_0\ldots\forall x_n 
(X(x_0,\ldots,x_n)\,\lra\,Y(x_0,\ldots,x_n)),
$$
i.e., the system $(\T,\vdash)$ proves that the formula 
$\theta$ defines a~single $(n+1)$-ary relation. 
Clearly, this happens iff there exists 
a~(second-order) formula $\psi(x_0,\ldots,x_n)$ 
having first-order parameters $x_0,\ldots,x_n$ 
and no other parameters and such that
$$
\T\;\vdash\;
\forall X
(\theta(X)\,\lra\,
\forall x_0\ldots\forall x_n 
(X(x_0,\ldots,x_n)\,\lra\,\psi(x_0,\ldots,x_n))).
$$

Uniqueness formulas in function variables are 
defined similarly. 

\hide{
IT SEEMS, NO, AT LEAST LITERALLY...

This second-order notion clearly absorbs the previous 
first-order one: if $\theta(x)$ is a~first-order 
uniqueness formula, then $\theta'(X)$ defined by 
$\forall x\,(X(x)\,\lra\,\theta(x))$ is a~second-order 
uniqueness formula ``semantically equivalent'' to $\theta$ 
in an obvious sense. Hence, when formulating 
the second-order restriction rules, it suffices 
to use only second-order uniqueness formulas. 
}
To simplify the notation, below we formulate only 
the non-relativized version of the second-order 
restriction rules.

Given a~set $\Theta$ of uniqueness formulas in 
a~second-order variable $X$ of a~fixed arity~$n$, 
the {\em $\Theta$-rule of arity~$n$} is the schema 
consisting of the following rules:
$$
\frac
{\{\varphi(X)\wedge\theta(X):\theta\in\Theta\}}
{\forall X\,\varphi(X)}.
$$ 
By the {\em $\Theta$-rule}, we mean 
the union of these schemas for all finite~$n$. 

Second-order provability involves, besides the usual 
first-order axioms and rules, the comprehension axioms 
and the generalization rule for second-order quantifiers. 
Considering this, it is easy to obtain a~second-order 
syntactic completeness theorem along with a~proof 
essentially repeating that of Theorem~\ref{t: syn complet}, 
where second-order uniqueness formulas act like additional 
predicate symbols; we leave the exact formulation and 
the proof to the reader.

Regarding semantics, recall that along with 
the {\em full} semantics, where the second-order 
quantifiers involve all relations and functions, 
there is the {\em general} or {\em Henkin} semantics
which uses only relations and functions in a~given class
(see, e.g.,~\cite{Vaananen 1979}).  
Obviously, if the universe of a~model~$\mathfrak A$ has the 
cardinality~$\kappa$ while the set $\Theta$ of second-order 
uniqueness formulas the cardinality~$<2^{\kappa}$ 
(e.g., if both the model and the language are countable), 
we have no guarantee to deduce the second-order theory 
of $\mathfrak A$ in the full semantics by means of 
the $\Theta$-rule. 
The correct version of Lemma~\ref{l: S-soundness} is: 
{\em The second-order $\Theta$-rule preserves validity 
in every model in the Henkin semantics based on~$\Theta$.}
With this in mind, a~second-order semantic completeness 
theorem analogous to Theorem~\ref{t: sem complet} can be 
obtained in the similar way. We leave the exact formulation 
and its proof to the reader and give here only two examples 
related to arithmetic.
\hide{
In this way, we can deduce (highly non-effectively, of 
course) from the second-order arithmetic~$\mathrm{Z}_2$
the true second-order arithmetic~$\TA_2$. 
}

\subsection*{An example in arithmetic.}

Below $\Z_2$ denotes, as usual, the extension of 
$\PA$ (or of the Robinson arithmetic~$\mathrm{Q}$) by 
the second-order induction and comprehension schemas. 
The set is {\em analytical} iff it is definable by 
a~formula of~$\Z_2$. $\TA_2$ is {\em second-order 
true arithmetic}, i.e., the complete second-order 
theory of the standard model of arithmetic.

Let also $\tau$ be the vocabulary of arithmetic expanded 
by symbols for all relations and operations on~$\omega$, 
and let $\Z_{2}(\tau)$ be the theory in $\tau$ 
extending $\Z_2$ by all the atomic and negated 
atomic sentences true in~$\omega$, i.e., if $R_S$ is 
the relation symbols for $S\subseteq\omega^m$ and $t_n$ 
is the term for~$n\in\omega$, then $\Z_{2}(\tau)$ 
contains $R_S(t_{n_0},\ldots,t_{n_{m-1}})$ 
if $(n_0,\ldots,n_{m-1})\in S$, and  
$\neg\,R_S(t_{n_0},\ldots,t_{n_{m-1}})$ otherwise. 
Let $\TA_2(\tau)$ denote the complete second-order 
theory of the standard model of arithmetic in~$\tau$.
By {\em true theory of analytical sets} we mean 
the theory of second-order arithmetic with 
the Henkin semantic based on analytical sets.

\hide{
\begin{prop}
Let $\Theta$ consist of uniqueness formulas of 
$\Z_2$. Then 
$$
[\Z_2,\vdash_\Theta]=\text{ the true theory of 
analytical sets.}
$$
\end{prop}

\begin{prop}
Let $\Theta$ consist of uniqueness formulas of 
$\Z_{2}(\tau)$. Then 
$$
[\Z_{2}(\tau),\vdash_\Theta]=\TA_2(\tau),
$$
and a fortiori, 
$\Z_{2}(\tau)\vdash_\Theta\TA_2$.
\end{prop}
}

\medskip
\begin{prop}~
\begin{enumerate}
\item[(i)] 
If $\Theta$ consists of uniqueness formulas 
of~$\Z_2$, then $[\Z_2,\vdash_\Theta]$ is 
true theory of analytical sets. 
\item[(ii)] 
If $\Theta$ consists of uniqueness formulas 
of $\Z_{2}(\tau)$, then 
$[\Z_{2}(\tau),\vdash_\Theta]=\TA_2(\tau),$
and a fortiori, $\Z_{2}(\tau)\vdash_\Theta\TA_2$.
\end{enumerate}
\end{prop}

\hide{
\begin{q}
What is complexity of the ``second-order generalized 
diagram'' (which is essentially the set of all 
second-order formulas in one-parameter)? 
Is it less complex, in a~sense, than $\TA_2$?
\end{q}
}


\begin{rmk} 
Let us notice a~few things in this connection.

1. 
The generalization of Post's theorem to 
the analytical hierarchy shows that $\TA_2$ is not 
analytical, i.e., not definable by any single 
second-order formula of arithmetic.
\hide{
Simpson (1977) has shown that the true theory of 
second-order arithmetic is computably interpretable with 
the theory of the partial order of all Turing degrees, 
in the signature of partial orders, and vice versa.
[Moreover, the true theory of analytical sets 
is not analytical. Clarify!]
}

2. 
$\TA_2$ and even the theory of analytical sets depends on 
the meta-theory. E.g., the projective determinacy $\PD$ is  
expressible as a~schema in the language of~$\mathrm{Z}_{2}$; 
the continuum hypothesis $\CH$ also is expressible by 
a~second-order sentence, in fact, by a~$\Pi^{1}_{1}$-sentence 
in the empty vocabulary.

3. 
Even validity in second-order logic is not second-order 
definable over $\omega$ nor over any second-order 
characterizable structure (the latters include 
$\mathbb R$, $\omega_1$, the first inaccessible 
cardinal, etc.), see~\cite{Vaananen 2012}. 
This follows from the two facts: 
the set of (G{\"o}del's numbers of) valid second-order 
sentences is a~complete $\Pi_2$-set \cite{Tharp 1973}
while every class of models that is definable by 
a~second-order formula is a~$\Delta_2$-set 
\cite{Vaananen 1979} (where $\Pi_2$ and $\Delta_2$ 
relate to the L{\'e}vy hierarchy in $\ZFC$).

4. 
As shown in~\cite{Montague} (see also 
\cite{Vaananen 1979,Vaananen 2012}), 
third-order logic, as well as much higher-order 
logics, is expressible in second-order logic.

5. 
By \cite{Shoenfield}, the recursively restricted 
$\omega$-rule suffices to deduce $\TA$ from~$\PA$. 
Of course, the $\omega$-rule does not deduce $\TA_2$ 
from~$\Z_2$, and Shoenfield asked whether the recursively 
restricted $\omega$-rule deduces the same fragment of 
$\TA_2$ as the full $\omega$-rule; the affirmative 
answer was given in~\cite{Takahashi 1970}. 
\end{rmk}


Further extensions of this machinery to higher-order 
infinitary logics are possible.

\vskip+1em


We conclude this note with a~few more questions.

\begin{q} 
Find restriction rules suitable for arbitrary 
model-theoretic logics in the sense of Barwise 
(see~\cite{Barwise Feferman}).
\end{q}

\begin{q}
Find ``more effective'' but equally strong versions 
of the proposed restriction rules. E.g., what can be 
analogs of the constructive version of the $\omega$-rule 
suitable for getting $\Th(L)$ 
or $\TA_2$? 
\end{q}


\begin{q}
What are modal logics arising, in the way described above, 
in higher model-theoretic logics with restriction rules? 
\end{q}



\vskip+2em

\begin{footnotesize} 
\noindent
{\sc
The Russian Academy of Sciences, 
Steklov Mathematical Institute, 
Gubkina street~8, 
Moscow 119991 Russia
\/}
\\
{\it E-mail address:\/} 
d.i.saveliev@gmail.com 
\end{footnotesize}

\newpage

\section*{Appendix: an overview of earlier results}

Here we restate some results of 
\cite{Henkin 1954,Henkin 1957,Rosser 1937,
Orey 1956,Barwise 1975} 
in terms slightly closer to ours.

\subsection*{
On L.~Henkin's ``A~generalization of 
the concept of $\omega$-consistency'', 
\cite{Henkin 1954}.} 

Let $S$ be a~set of constants.

$\T$~is $S$-{\em consistent} iff there is no 
$\varphi(x)$ such that 
$\T\vdash\varphi(a)$ for all $a\in S$ and 
$\T\vdash\exists x\,\neg\,\varphi(x)$.

$\T$~is $S$-{\em satisfiable} iff there is  
a~model satisfying $\T$ and omitting the $1$-type 
$\{x\ne a:a\in S\}$ (i.e., the universe of which 
consists of interpretations of the constants in~$S$).

Facts: 
If $\T$~is $S$-satisfiable, then it is $S$-consistent 
(Theorem~1). 
If $\T$~is $S$-consistent and $|S|<\omega$, 
then it is $S$-satisfiable (Theorem~2). 
There are $\T$ and $S$ with $|S|=\omega$ such that 
$\T$~is $S$-consistent but not $S$-satisfiable 
(Theorem~3). 

$\T$~is $S^{(n)}$-{\em consistent} iff there is 
no $\varphi(x_0,\ldots,x_{n-1})$ such that 
$\T\vdash\varphi(a_0,\ldots,a_{n-1})$ for all 
$a_i\in S$ and 
$
\T\vdash\exists x_0\ldots\exists x_{n-1}
\,\neg\,\varphi(x_0,\ldots,x_{n-1})
$, 
and 
$S^{(\omega)}$-{\em consistent} iff it is 
$S^{(n)}$-consistent for all $n<\omega$.

Facts: 
For all $n<\omega$, there is $\T$ which is 
$S^{(n+1)}$- but not $S^{(n+2)}$-consistent 
(Theorem~4). 
If $\T$~is $S$-satisfiable, then 
it is $S^{(\omega)}$-consistent (Theorem~5). 
There are $\T$ and $S$ with $|S|=\omega$ such that 
$\T$~is $S^{(\omega)}$-consistent but not 
$S$-satisfiable (Theorem~6).

$\T$~is {\em strongly $S$-consistent} iff there is 
a~mapping of formulas in one variable into~$S$, 
say $\varphi\mapsto a_\varphi$, such that 
there is no $\varphi(x)$ such that 
$
\T\vdash
\bigvee_{i<n}
(\varphi_i(a_{\varphi_i})\wedge 
\exists x\,\neg\,\varphi_i(x)).
$

$\T$~is $S$-satisfiable iff 
it is strongly $S$-consistent (Theorem~7).

\subsection*{
On L.~Henkin's ``A~generalization of 
the concept of $\omega$-completeness'', 
\cite{Henkin 1957}.} 

As before, $S$~is a~fixed set of constants.

$\T$~is $S$-{\em complete} iff for all 
$\varphi(x)$, 
$\T\vdash\varphi(a)$ for all $a\in S$ implies 
$\T\vdash\forall x\,\varphi(x)$ 
(i.e., $[\T,\vdash]=[\T,\vdash_S]$, 
$\T$~is ``closed under $S$-rule'').
Notice: the $S^{(n)}$-completeness gives the same.

$\T$~is $S$-{\em saturated} iff 
$\T\vdash\sigma$ whenever $\sigma$~is true 
in all models that $S$-satisfy~$\T$. 
(Warning: do not confuse with saturated models!)

Facts: 
If $\T$~is $S$-saturated, then it is $S$-complete 
(Theorem~1). 
If $\T$~is $S$-complete and $|S|<\omega$, 
then it is $S$-saturated (Theorem~2). 
Moreover, in countable languages,
if $\T$~is $S$-complete and $|S|=\omega$, 
then it is $S$-saturated (Theorem~3). 
There are $\T$ and $S$ with $|S|>\omega$ such that 
$\T$~is $S$-complete but not $S$-saturated 
(Theorem~4); moreover, such a~$\T$ can be assumed 
to be $S$-satisfiable (Theorem~5).

$\T$~is {\em strongly $S$-complete} iff there is 
a~mapping of pairs $(\sigma,\varphi)$, where 
$\sigma$~is a~sentence and $\varphi(x)$~a~formula 
in one variable, into~$S$, 
say $(\sigma,\varphi)\mapsto a_{\sigma,\varphi}$, 
such that the set 
$
\Delta_\sigma:=
\{
(\sigma\wedge\exists x\,\varphi(x))\to
\varphi(a_{\sigma,\varphi}):\varphi(x)
\}
$ 
is consistent in~$\T$, for all~$\sigma$.

$\T$~is consistent and $S$-saturated iff 
it is strongly $S$-complete (Theorems 6 and~7).

\subsection*{
On S.~Orey's 
``On $\omega$-consistency and related properties'', 
\cite{Orey 1956},  
and related facts in B.~Rosser's 
``G{\"o}del theorems for non-constructive logics'', 
\cite{Rosser 1937}.
} 

\par
Theorem~2 of~\cite{Orey 1956} is essentially 
Theorem~3 of~\cite{Henkin 1954}.  

According to~\cite{Rosser 1937}, 
$\T$~is {\em $\omega(\alpha)$-consistent} iff 
it is closed under $\alpha$~applications of 
the $\omega$-rule. Therefore, 
$\T$~is $\omega(\alpha)$-consistent with 
$\alpha=\omega_1$ iff it is $\omega$-complete.

$\T$~is $\omega(\omega_1)$-consistent iff 
it is $\omega$-satisfiable (Theorems 3 and~4 
of~\cite{Orey 1956}, also this follows from 
Theorems 1 and~3 of~\cite{Henkin 1957}). 

(Note that in these statements, $\omega_1$ can 
be replaced with $\omega^{\CK}_1$.) 

\subsection*{
On related results in J.~Barwise's 
``Admissible sets and structures'',  
\cite{Barwise 1975}, Ch.~III,~\S3. 
}

Let $\mathfrak A$ be a~model of a~countable 
vocabulary~$\tau$, and let another vocabulary $\tau^+$ 
include $\tau$ and have a~unary predicate symbol~$U$ 
and constant symbols~$c_a$ for all $a\in A$. 
A~model~$\mathfrak A^+$ of $\tau^+$ is 
an {\em $\mathfrak A$-model} iff it interprets 
$U$ by~$A$, each $c_a$ by~$a$, and $\mathfrak A$ is 
a~submodel of the restriction of $\mathfrak A^+$ 
to~$\tau$.

Let $\T\vDash_{\mathfrak A}\varphi$ iff 
$\mathfrak A^+\vDash\varphi$ for all $\mathfrak A$-models 
$\mathfrak A^+$ that satisfy~$\T$, and let  
$\T\vdash_{\mathfrak A}\varphi$ iff $\varphi$ is 
deducible from $\T$ in the {\em $\mathfrak A$-logic}, 
the axioms of which consist of the usual logical axioms 
extended by $\{U(c_a):a\in A\}$ and $\Dg(\mathfrak A)$, 
and the rules are MP, Generalization, and 
the $S$-rule with $S:=\{c_a:a\in A\}$.

Facts: 
$\T\vdash_{\mathfrak A}\varphi(x_0,\ldots,x_{n-1})$ 
implies 
$
\T\vDash_{\mathfrak A}
\forall x_0\ldots\forall x_{n-1}\,
\varphi(x_0,\ldots,x_{n-1})
$
(i.e., the $\mathfrak A$-logic is sound). 
Moreover, if $|\tau^+|\le\omega$ (and so $|A|\le\omega$), 
then the converse implication holds as well, thus giving 
the {\em $\mathfrak A$-Completeness Theorem}:
$\T\vdash_{\mathfrak A}\varphi$ iff 
$\T\vDash_{\mathfrak A}\varphi$ for all 
sentences~$\varphi$ (Theorem~3.5).

\subsection*{
A characterization of $\omega$-consistency 
in C.~Smory{\'n}ski's 
``Self-reference and modal logic'', 1985, 
\cite{Smorynski 1985}.  
}

Let 
$\TA_{{\Pi_{2}^{0}}}$ be the $\Pi_{2}^{0}$-fragment 
of $\TA$, i.e., the set of all $\Pi^{0}_2$-sentences 
valid in the standard model of arithmetic, and 
${\RFN}_{\T}$ the uniform reflection principle for~$\T$, 
i.e., the schema of the axioms
$
\forall x\,(\Pr_{\T}
(\ulcorner\varphi({\dot x})\urcorner)\to\varphi(x))
$
for every formula $\varphi$ in one free variable. 

Fact: 
A~recursively axiomatizable theory $\T$ 
is $\omega$-consistent iff the theory 
$\T+\TA_{{\Pi_{2}^{0}}}+\RFN_\T$ is consistent. 
In particular, a~finitely axiomatizable theory $\T$ 
in the language of arithmetic is $\omega$-consistent 
iff $\T+\PA$ is $\Sigma_{2}^{0}$-sound.

\hide{

Cohen (1963) gave another construction of the minimal 
model as the strongly constructible sets, using 
a modified form of Gödel's constructible universe.

...

Paris: initiated the study of
DO models and showed that 
(1) every consistent extension 
T of ZF has a DO model, and
(2) for complete extensions~T, 
T has a unique DO model up to isomorphism 
iff T proves V = OD. 

Enayat:
1. If T is a consistent completion 
of $\ZF+V\ne OD$, then T has continuum-many 
countable nonisomorphic Paris models.
2. Every countable model of ZFC has 
a Paris generic extension.
3. If there is an uncountable well-founded 
model of ZFC, then for every infinite 
cardinal~$\kappa$ there is a Paris model of ZF 
of cardinality~$\kappa$ which has 
a nontrivial automorphism.
4. For a model $\mathfrak M\vDash\ZF$, 
$\mathfrak M$ is a~prime model implies that 
$\mathfrak M$ is a Paris model and satisfies~$\AC$,
which in turn implies that
$\mathfrak M$~is a~minimal model;
neither implication reverses assuming Con(ZFC).

Hamkins et al:

A pointwise definable model is one in which 
every object is definable without parameters.
In a model of set theory, this property strengthens 
V = HOD, but is not first-order expressible. 
Nevertheless, if ZFC is consistent, then there 
are continuum many pointwise definable models of ZFC. 
If there is a transitive model of ZFC, then 
there are continuum many pointwise definable 
transitive models of ZFC. 
What is more, every countable model of ZFC has 
a class forcing extension that is pointwise definable. 
Indeed, every countable model of Godel--Bernays 
set theory has a pointwise definable extension, 
in which every set and class is first-order definable 
without parameters.

It is well known that the first order arithmetic 
with the $\omega$-rule is complete. Moreover, 
J. R. Shoenfield [3] has shown that the same holds 
also when the $\omega$-rule is recursively restricted. 
On the other hand, the second-order arithmetic with 
the $\omega$-rule is not complete. So Shoenfield has 
raised a question whether every sentence of the 
second-order arithmetic provable with the $\omega$-rule 
is provable with the recursively restricted $\omega$-rule.
Takahashi [1] gave an affirmative answer.

[What is an easiest instance of a~true second-order 
sentence unprovable with the $\omega$-rule? Probably, 
the sentence expressing consistency of the logic with 
the $\omega$-rule (which is second-order formalizable).]

}

\end{document}